\documentclass[11pt,a4paper]{article}
\usepackage{amsmath,amssymb,amsfonts,amsthm,mathtools}
\usepackage{bm}
\usepackage{graphicx}
\usepackage{graphicx}
\usepackage{subcaption}
\usepackage{booktabs}
\usepackage{multirow}
\usepackage{array}
\usepackage{enumitem}
\usepackage{geometry}
\usepackage{setspace}
\usepackage{cite}
\usepackage{url}
\usepackage{hyperref}
\usepackage{xcolor}
\usepackage{ragged2e}
\hypersetup{
    colorlinks=true,
    linkcolor=blue,
    citecolor=blue,
    urlcolor=blue
}

\newtheorem{theorem}{Theorem}[section]

\newcommand{\Rzero}{\mathcal{R}_0}

\begin{document}

\title{\textbf{Threshold Geometry, Bifurcation, and Data-Driven Analysis of a Vaccination-Treatment Model for Hepatitis B}}

\author{
    Mustaq Ahmad$^{1}$
    \and
    A. S. Bhadauria$^{1}$\thanks{Corresponding author: archana.mathstat@ddugu.ac.in}
}

\date{}

\maketitle


\begin{abstract}
Despite the availability of excellent vaccines and antiviral medicines, hepatitis B virus (HBV) remains a major public health problem. In this work, we develop a mathematical vaccination-treatment model for HBV transmission to investigate threshold conditions for illness persistence as well as the level of intervention. The analysis gives the characterisation of both disease-free and endemic equilibria and determines the basic reproduction number, $\mathcal{R}_0$, by the next-generation matrix approach. The model indicates that $\mathcal{R}_0=1$ is a critical threshold dividing disease-free from endemic states and characterises a forward transcritical bifurcation regarding the transmission parameter. Numerical equilibrium continuation and stability assessments confirm the analytical results. A two-parameter study of transmission and vaccination rates identifies threshold geometries and assesses impacts of important parameters using sensitivity analysis. The model is calibrated with India-specific data, such as HBsAg prevalence, HBV incidence and mortality rates and fits epidemiological targets closely through limited nonlinear calibration. Additionally, multi-start optimisation and practical identifiability profiling are used to investigate parameter uncertainty. Robustness analyses reveal that the calibrated $\mathcal{R}_0^*=1.1744$ is associated with a super-threshold regime; however, parameter flexibility permits both sub-threshold and super-threshold regimes. This combined analytical and data-driven framework connects epidemic thresholds, bifurcation structures, intervention efficacy, and parameter uncertainty. It provides a quantitative framework for analysing HBV persistence and evaluating control options, especially in the absence of epidemiological evidence.
\medskip
\noindent

\end{abstract}

\section{Introduction}
\label{sec_introduction}

Hepatitis B virus (HBV) \cite{who2024global,world2024guidelines,worldint,world2026global} infection remains a major global public health concern because of its potential to cause both acute and chronic disease and to progress to severe liver-related complications, including cirrhosis and hepatocellular carcinoma. Despite the availability of effective vaccination and antiviral treatment, HBV continues to persist in many populations, particularly in settings where vaccination coverage, timely diagnosis, and access to treatment remain heterogeneous. The persistence of HBV is governed by a combination of epidemiological and demographic processes, including transmission, population recruitment, vaccination, treatment, and disease-associated removal. A quantitative understanding of the interaction among these mechanisms is therefore important for identifying the conditions under which HBV can be controlled or persist within a population.

Mathematical models provide a systematic framework for investigating infectious disease dynamics and for identifying threshold conditions associated with disease invasion and persistence. Compartmental models \cite{anderson1991infectious,martcheva2015introduction,hethcote2000mathematics} are particularly useful because they represent the movement of individuals between epidemiological states and allow the effects of prevention and treatment mechanisms to be incorporated explicitly. A central quantity in epidemic modeling is the basic reproduction number $\mathcal{R}_0$ \cite{van2002reproduction,diekmann2010construction,diekmann1990definition}, which represents the expected number of secondary infections generated by a typical infectious individual introduced into an otherwise disease-free population. In the classical threshold framework, $\mathcal{R}_0<1$ is associated with local disease elimination, whereas $\mathcal{R}_0>1$ permits invasion and persistence.

For HBV, vaccination and treatment act through complementary mechanisms  \cite{al2024global}. Vaccination reduces the number of susceptible individuals, whereas treatment reduces the infectious contribution of infected individuals and can consequently alter the conditions required for continued transmission. The combined effects of these interventions can therefore generate nontrivial threshold behavior. In particular, the epidemic threshold need not depend on transmission intensity alone; it may also be determined by interactions between transmission, vaccination, treatment, and other demographic or disease-related parameters \cite{easterbrook20242024}.

Although several mathematical models have been developed to investigate HBV transmission and the effects of vaccination and treatment, many analyses focus primarily on the value of $\mathcal{R}_0$ at a fixed parameter set or examine individual parameters independently \cite{mann2011modelling,g_i__marchuk__1991,ranchov_gk_1996,zaman2008stability,GUMEL2003409,zou2010modeling,seeger2000hepatitis,thornley2008hepatitis}. Such approaches may provide useful threshold information but do not fully characterize the geometry of the parameter space near the epidemic boundary. In particular, a threshold condition such as $\mathcal{R}_0=1$ can be viewed not only as a scalar criterion but also as a boundary separating qualitatively different dynamical regimes in a multidimensional parameter space. Furthermore, the local behavior of equilibria near this threshold provides important information about how disease-free and endemic states exchange stability. A rigorous bifurcation analysis is therefore valuable for determining the nature of the transition between disease elimination and endemic persistence.

Motivated by these considerations, we formulate a vaccination-treatment model for HBV transmission and investigate its threshold and bifurcation structure. The model incorporates susceptible, vaccinated, infected, and treated populations, together with demographic recruitment and disease-induced removal. We establish the positivity and boundedness of solutions, characterize the disease-free and endemic equilibria, and derive the basic reproduction number using the next-generation matrix approach. We then examine the local stability \cite{hurwitz1964conditions,castillo2002mathematical,CastilloChavezSong2004,patil2021routh} properties of the equilibria and investigate the threshold condition $\mathcal{R}_0=1$ analytically through center-manifold theory \cite{wang2004bifurcations,wang2006backward}. In particular, we establish the occurrence of a forward transcritical bifurcation with respect to the transmission parameter, thereby providing a rigorous characterization of the local transition between disease-free and endemic states.

Beyond the one-parameter threshold condition, we investigate the interaction between transmission and vaccination through a two-parameter analysis in the $(\beta,\nu_1)$ parameter plane. The resulting curve $\mathcal{R}_0(\beta,\nu_1)=1$ provides a geometric boundary separating sub-threshold and super-threshold regimes. The corresponding endemic-equilibrium surface is used to examine how the magnitude of persistent infection varies across the same parameter space. Local normalized sensitivity indices are further employed to quantify the influence of model parameters on $\mathcal{R}_0$, while a global sensitivity analysis based on Latin hypercube sampling and partial rank correlation coefficients (PRCCs) \cite{qian2020sensitivity,gomero2012latin} is used to assess the effects of parameter uncertainty on both the reproduction number and endemic infection level.

An additional objective of this study is to connect the analytical model with epidemiological observations. Accordingly, the model is parameterized using India-specific demographic and epidemiological information, including HBsAg prevalence, HBV incidence, and mortality estimates \cite{world2020hepatitis}. A constrained nonlinear calibration procedure is employed to identify parameter combinations capable of reproducing the prescribed epidemiological targets. The calibrated model is subsequently evaluated through direct model-data comparison, while multi-start optimization is used to examine the robustness of the calibration with respect to the initial parameter guesses. Because an accurate fit to aggregate epidemiological observations does not necessarily imply unique determination of the underlying mechanistic parameters, a profile-based practical identifiability analysis is also performed. Finally, the robustness of the inferred basic reproduction number is examined over the practically compatible parameter space to determine whether the threshold classification associated with the reference calibration is preserved under parameter uncertainty.

The main objectives of this study are therefore to (i) formulate a vaccination-treatment model for HBV transmission and establish its positivity, boundedness, equilibrium structure, and basic reproduction number; (ii) characterize the local stability of the disease-free and endemic equilibria and establish the direction of the bifurcation occurring at the epidemic threshold; (iii) investigate the interaction between transmission and vaccination through a two-parameter threshold geometry and quantify the sensitivity of $\mathcal{R}_0$ and the endemic infection level to key parameters; (iv) calibrate the model using India-specific epidemiological targets and evaluate the resulting model-data agreement; and (v) assess calibration robustness, practical parameter identifiability, and the robustness of $\mathcal{R}_0$ over the parameter region compatible with the available epidemiological observations.

The remainder of the paper is organized as follows. Section \ref{sec_model} presents the model formulation and underlying assumptions. Section \ref{sec_properties} establishes the basic mathematical properties of the model, such as the existence, positivity, and boundedness; characterizes the disease-free and endemic equilibria; and derives the basic reproduction number $\mathcal{R}_0$. Section \ref{sec_Stability} investigates the local stability of the disease-free equilibrium and the endemic equilibrium. Section \ref{sec_bifurcation} presents the analytical forward transcritical bifurcation analysis. Section \ref{sec_two_parameter} develops the two-parameter bifurcation analysis in the $(\beta,\nu_1)$-plane, and Section \ref{sec_threshold} establishes the corresponding analytical threshold structure. Section \ref{sec_numerical} presents the numerical results, including equilibrium and bifurcation diagrams and sensitivity analyses. Section \ref{sec_data_parameterization} develops the data-driven parameterization and validation of the model, including calibration, model-data comparison, multi-start calibration, practical identifiability, and robustness of the basic reproduction number. Section \ref{sec_discussion} discusses the epidemiological implications and limitations of the results, while Section \ref{sec_conclusion} concludes the study.

\section{Model formulation}
\label{sec_model}

\subsection{Compartmental structure}
\label{subsec_compartments}

Let the total population at time $t$ be
\begin{equation}
N(t)=S(t)+I(t)+R(t).
\end{equation}
The susceptible class $S(t)$ contains individuals who are susceptible to HBV infection. Individuals enter this class through recruitment, may acquire infection through horizontal transmission, and may leave it through vaccination or natural mortality. The infected class $I(t)$ consists of individuals infected with HBV. Infected individuals may receive treatment and subsequently enter the recovered class. They may also experience natural mortality or disease-induced mortality. The recovered class $R(t)$ contains individuals who have acquired immunity through vaccination or treatment. Since immunity is assumed to be permanent, recovered individuals do not return to the susceptible compartment.

The principal transitions are
\begin{equation}
S \xrightarrow{\beta SI} I, \qquad S \xrightarrow{\nu_1} R, \qquad I \xrightarrow{\nu_2} R.
\end{equation}
Natural mortality occurs in all three compartments, whereas disease-induced mortality occurs only in the infected class.

\subsection{Model assumptions}
\label{subsec_assumptions}

The model is based on the following assumptions:
\begin{enumerate}
\item[(i)] the population is homogeneously mixed;
\item[(ii)] individuals enter the population at a constant rate $\lambda$;
\item[(iii)] HBV is transmitted horizontally through the bilinear incidence $\beta SI$;
\item[(iv)] susceptible individuals are vaccinated at rate $\nu_1$;
\item[(v)] infected individuals receive treatment at rate $\nu_2$;
\item[(vi)] infected individuals experience disease-induced mortality at rate $\sigma$;
\item[(vii)] natural mortality occurs at rate $\mu$ in all compartments;
\item[(viii)] vaccination and treatment result in movement into the recovered class;
\item[(ix)] immunity is permanent;
\item[(x)] vertical transmission is neglected; and
\item[(xi)] the transmission coefficient $\beta$ is constant.
\end{enumerate}

Under these assumptions, the mathematical model is given by
\begin{equation}
\begin{aligned}
\frac{dS}{dt} &= \lambda-\beta SI-(\nu_1+\mu)S,\\
\frac{dI}{dt} &= \beta SI-(\sigma+\nu_2+\mu)I,\\
\frac{dR}{dt} &= \nu_1S+\nu_2I-\mu R.
\end{aligned}
\label{eq:model}
\end{equation}

Here, $S(t)$, $I(t)$, and $R(t)$ denote the susceptible, infected, and recovered populations, respectively. The parameter $\lambda$ is the recruitment rate, $\beta$ is the horizontal transmission coefficient, $\nu_1$ is the vaccination rate, $\nu_2$ is the treatment rate, $\sigma$ is the disease-induced mortality rate, and $\mu$ is the natural mortality rate.

\section{Basic mathematical properties and equilibrium analysis}
\label{sec_properties}

For convenience, define
\begin{equation}
A=\nu_1+\mu,
\qquad
Q=\sigma+\nu_2+\mu.
\label{eq:AQ}
\end{equation}
Then system~\eqref{eq:model} can be written as
\begin{equation}
\begin{aligned}
\dot S &= \lambda-\beta SI-AS,\\
\dot I &= \beta SI-QI,\\
\dot R &= \nu_1S+\nu_2I-\mu R.
\end{aligned}
\label{eq:compact_model}
\end{equation}

\subsection{Existence, positivity and boundedness}
\label{subsec_existence}

Let
\begin{equation}
X(t)=(S(t),I(t),R(t))^{T}.
\end{equation}
Then system~\eqref{eq:compact_model} can be expressed as
\begin{equation}
\dot X=F(X),
\end{equation}
where $F:\mathbb{R}^{3}\rightarrow\mathbb{R}^{3}$ is continuously differentiable. Hence, $F$ is locally Lipschitz, and for every nonnegative initial condition
\begin{equation}
S(0)=S_0\geq0,\qquad I(0)=I_0\geq0,\qquad R(0)=R_0\geq0,
\end{equation}
system~\eqref{eq:compact_model} admits a unique local solution.

We next establish positivity. On the boundary $S=0$,
\begin{equation}
\left.\frac{dS}{dt}\right|_{S=0}=\lambda>0.
\end{equation}
On the boundary $I=0$,
\begin{equation}
\left.\frac{dI}{dt}\right|_{I=0}=0,
\end{equation}
while on the boundary $R=0$,
\begin{equation}
\left.\frac{dR}{dt}\right|_{R=0}=\nu_1S+\nu_2I\geq0.
\end{equation}
Therefore, the vector field is inward-pointing or tangent on each coordinate boundary of $\mathbb{R}^{3}_{+}$. Consequently, the nonnegative orthant is positively invariant and
\begin{equation}
S(t)\geq0,\qquad I(t)\geq0,\qquad R(t)\geq0,
\qquad t\geq0.
\end{equation}

To establish boundedness, let
\begin{equation}
N(t)=S(t)+I(t)+R(t).
\end{equation}
Adding the equations in~\eqref{eq:compact_model} gives
\begin{equation}
\frac{dN}{dt}=\lambda-\mu N-\sigma I
\leq\lambda-\mu N.
\end{equation}
By the comparison theorem,
\begin{equation}
N(t)\leq N(0)e^{-\mu t}
+\frac{\lambda}{\mu}\left(1-e^{-\mu t}\right),
\end{equation}
and therefore
\begin{equation}
N(t)\leq\max\left\{N(0),\frac{\lambda}{\mu}\right\}.
\end{equation}
Thus, every nonnegative solution is uniformly bounded and the local solution extends globally.

\subsection{Biologically feasible region}
\label{subsec:feasible}

The positivity and boundedness results imply that the region
\begin{equation}
\Omega=
\left\{
(S,I,R)\in\mathbb{R}^{3}_{+}:
S+I+R\leq\frac{\lambda}{\mu}
\right\}
\label{eq:Omega}
\end{equation}
is positively invariant whenever
\begin{equation}
N(0)\leq\frac{\lambda}{\mu}.
\end{equation}
Hence, $\Omega$ is a biologically feasible region for system~\eqref{eq:compact_model}.

\subsection{Disease-free equilibrium}
\label{subsec_dfe}

At the disease-free equilibrium, $I=0$. Therefore,
\begin{equation}
0=\lambda-AS,
\end{equation}
which gives
\begin{equation}
S_0=\frac{\lambda}{A}.
\end{equation}
The third equilibrium equation yields
\begin{equation}
R_0=\frac{\nu_1S_0}{\mu}
=\frac{\lambda\nu_1}{\mu A}.
\end{equation}
Hence, the disease-free equilibrium is
\begin{equation}
E_0=
\left(
\frac{\lambda}{\nu_1+\mu},
0,
\frac{\lambda\nu_1}{\mu(\nu_1+\mu)}
\right).
\label{eq:dfe}
\end{equation}

\subsection{Basic reproduction number}
\label{subsec_R0}

The infected equation in~\eqref{eq:compact_model} can be written as
\begin{equation}
\dot I=\beta SI-QI.
\end{equation}
At the disease-free equilibrium,
\begin{equation}
S=S_0=\frac{\lambda}{A}.
\end{equation}
The new-infection and transition terms are therefore
\begin{equation}
\mathcal{F}=\beta S_0 I,
\qquad
\mathcal{V}=QI.
\end{equation}
Using the next-generation matrix approach, the basic reproduction number is
\begin{equation}
\mathcal{R}_0=\frac{\beta S_0}{Q},
\end{equation}
and hence
\begin{equation}
\boxed{
\mathcal{R}_0=
\frac{\beta\lambda}
{(\nu_1+\mu)(\sigma+\nu_2+\mu)}
}.
\label{eq:R0}
\end{equation}
Thus, $\mathcal{R}_0$ provides the fundamental threshold quantity governing HBV invasion and persistence.

\subsection{Endemic equilibrium}
\label{subsec_endemic}

Let
\begin{equation}
E^*=(S^*,I^*,R^*)
\end{equation}
denote an endemic equilibrium with $I^*>0$. From the infected equilibrium equation,
\begin{equation}
0=I^*(\beta S^*-Q).
\end{equation}
Since $I^*>0$,
\begin{equation}
S^*=\frac{Q}{\beta}
=\frac{\sigma+\nu_2+\mu}{\beta}.
\label{eq:Sstar}
\end{equation}
Substitution into the susceptible equilibrium equation gives
\begin{equation}
I^*
=\frac{\lambda}{Q}-\frac{A}{\beta}
=\frac{A}{\beta}(\mathcal{R}_0-1).
\label{eq:Istar}
\end{equation}
Finally,
\begin{equation}
R^*
=\frac{\nu_1S^*+\nu_2I^*}{\mu},
\end{equation}
so that
\begin{equation}
R^*
=
\frac{\nu_1Q+\nu_2A(\mathcal{R}_0-1)}
{\mu\beta}.
\label{eq:Rstar}
\end{equation}
Therefore,
\begin{equation}
E^*=
\left(
\frac{Q}{\beta},
\frac{A}{\beta}(\mathcal{R}_0-1),
\frac{\nu_1Q+\nu_2A(\mathcal{R}_0-1)}
{\mu\beta}
\right).
\label{eq:endemic}
\end{equation}

Since $A>0$ and $\beta>0$,
\begin{equation}
I^*>0
\quad\Longleftrightarrow\quad
\mathcal{R}_0>1.
\end{equation}
Hence, the endemic equilibrium exists in the biologically feasible region if and only if $\mathcal{R}_0>1$. At $\mathcal{R}_0=1$, we have $I^*=0$, and the endemic and disease-free equilibria coincide.

\section{Local stability analysis}
\label{sec_Stability}

The local stability of the disease-free and endemic equilibria is determined from the Jacobian matrix of system~\eqref{eq:compact_model}. The Jacobian matrix is
\begin{equation}
J(S,I,R)=\begin{pmatrix}
-\beta I-A & -\beta S & 0\\
\beta I & \beta S-Q & 0\\
\nu_1 & \nu_2 & -\mu
\end{pmatrix},
\label{eq:jacobian}
\end{equation}

\subsection{Local stability of the disease-free equilibrium}
\label{subsec:dfe_stability}

At the disease-free equilibrium
\begin{equation}
E_0=\left(\frac{\lambda}{A},0,\frac{\lambda\nu_1}{\mu A}\right),
\end{equation}
the Jacobian~\eqref{eq:jacobian} becomes
\begin{equation}
J(E_0)=
\begin{pmatrix}
-A & -\dfrac{\beta\lambda}{A} & 0\\
0 & \dfrac{\beta\lambda}{A}-Q & 0\\
\nu_1 & \nu_2 & -\mu
\end{pmatrix}.
\end{equation}
Since
\begin{equation}
\frac{\beta\lambda}{A}=Q\mathcal{R}_0,
\end{equation}
the eigenvalues of $J(E_0)$ are
\begin{equation}
\lambda_1=-A,\qquad
\lambda_2=Q(\mathcal{R}_0-1),\qquad
\lambda_3=-\mu.
\end{equation}
Because $A>0$, $Q>0$, and $\mu>0$, it follows that all eigenvalues have negative real parts when $\mathcal{R}_0<1$, whereas $\lambda_2>0$ when $\mathcal{R}_0>1$. Therefore, we obtain the following result.

\begin{theorem}
The disease-free equilibrium $E_0$ is locally asymptotically stable for $\mathcal{R}_0<1$ and unstable for $\mathcal{R}_0>1$.
\end{theorem}

\begin{proof}
The result follows directly from the eigenvalues of the Jacobian evaluated at $E_0$. For $\mathcal{R}_0<1$, all three eigenvalues are negative, while for $\mathcal{R}_0>1$, the eigenvalue $Q(\mathcal{R}_0-1)$ is positive.
\end{proof}

\subsection{Local stability of the endemic equilibrium}
\label{subsec_endemic_stability}

Consider the endemic equilibrium $E^*$, which exists for $\mathcal{R}_0>1$. From~\eqref{eq:Sstar} and~\eqref{eq:Istar}, we have
\begin{equation}
\beta S^*=Q,\qquad \beta I^*=A(\mathcal{R}_0-1).
\end{equation}
Define
\begin{equation}
X=A(\mathcal{R}_0-1),
\end{equation}
so that $X>0$ whenever $\mathcal{R}_0>1$. The Jacobian at $E^*$ is then
\begin{equation}
J(E^*)=
\begin{pmatrix}
-(A+X) & -Q & 0\\
X & 0 & 0\\
\nu_1 & \nu_2 & -\mu
\end{pmatrix}.
\end{equation}
The characteristic equation is
\begin{equation}
\det(\xi I-J(E^*))=(\xi+\mu)\left[\xi^2+(A+X)\xi+QX\right].
\end{equation}
Since $A+X=A\mathcal{R}_0$ and $X=A(\mathcal{R}_0-1)$, this becomes
\begin{equation}
P(\xi)=(\xi+\mu)\left[\xi^2+A\mathcal{R}_0\xi+AQ(\mathcal{R}_0-1)\right].
\end{equation}
Expanding the characteristic polynomial gives
\begin{equation}
P(\xi)=\xi^3+a_1\xi^2+a_2\xi+a_3,
\end{equation}
where
\begin{equation}
a_1=A\mathcal{R}_0+\mu,
\end{equation}
\begin{equation}
a_2=A\left[Q(\mathcal{R}_0-1)+\mu\mathcal{R}_0\right],
\end{equation}
and
\begin{equation}
a_3=\mu AQ(\mathcal{R}_0-1).
\end{equation}
For $\mathcal{R}_0>1$, all three coefficients are positive. Moreover,
\begin{equation}
a_1a_2-a_3=QA^2\mathcal{R}_0(\mathcal{R}_0-1)+\mu A\mathcal{R}_0(A\mathcal{R}_0+\mu)>0.
\end{equation}
Hence, all Routh-Hurwitz conditions for a cubic polynomial are satisfied.

\begin{theorem}
For $\mathcal{R}_0>1$, the endemic equilibrium $E^*$ is locally asymptotically stable.
\end{theorem}

\begin{proof}
For $\mathcal{R}_0>1$, we have $a_1>0$, $a_2>0$, and $a_3>0$, while $a_1a_2-a_3>0$. Therefore, the Routh-Hurwitz criterion implies that all eigenvalues of $J(E^*)$ have negative real parts. Hence, the endemic equilibrium $E^*$ is locally asymptotically stable.
\end{proof}

The local stability results can therefore be summarized by the threshold relation
\begin{equation}
\boxed{
\begin{aligned}
\mathcal{R}_0<1 &\quad\Longrightarrow\quad E_0\ \text{is locally asymptotically stable},\\
\mathcal{R}_0=1 &\quad\Longrightarrow\quad E_0=E^*,\\
\mathcal{R}_0>1 &\quad\Longrightarrow\quad E^*\ \text{exists and is locally asymptotically stable}.
\end{aligned}}
\label{eq:stability_summary}
\end{equation}
This threshold structure provides the basis for the bifurcation analysis in the next section.

\section{Forward transcritical bifurcation at \texorpdfstring{$\mathcal{R}_0=1$}{R0=1}}
\label{sec_bifurcation}

The stability results established in the previous section show that the disease-free equilibrium changes stability as the basic reproduction number crosses unity. Moreover, the disease-free and endemic equilibria coincide when $\mathcal{R}_0=1$. This indicates the possibility of a transcritical bifurcation at the epidemic threshold. To determine the direction of the bifurcation rigorously, we employ the center-manifold theorem with the transmission rate $\beta$ as the bifurcation parameter.

From the expression for the basic reproduction number,
\begin{equation}
\mathcal{R}_0=\frac{\beta\lambda}{AQ},
\end{equation}
the critical transmission rate is obtained by imposing $\mathcal{R}_0=1$. Hence,
\begin{equation}
\beta_c=\frac{AQ}{\lambda}=\frac{(\nu_1+\mu)(\sigma+\nu_2+\mu)}{\lambda}.
\label{eq:beta_critical}
\end{equation}

At $\beta=\beta_c$, the Jacobian evaluated at the disease-free equilibrium is
\begin{equation}
J_c=
\begin{pmatrix}
-A & -Q & 0\\
0 & 0 & 0\\
\nu_1 & \nu_2 & -\mu
\end{pmatrix}.
\label{eq:critical_jacobian}
\end{equation}
The eigenvalues of $J_c$ are
\begin{equation}
\lambda_1=-A,\qquad \lambda_2=0,\qquad \lambda_3=-\mu.
\end{equation}
Since $A>0$ and $\mu>0$, the critical Jacobian has exactly one simple zero eigenvalue, while the remaining two eigenvalues have negative real parts. Thus, the center subspace is one-dimensional, which is the appropriate setting for the application of the center-manifold theorem.

\subsection{Center-manifold coefficients}
\label{subsec_center_manifold}

Let $v$ and $w$ denote the right and left eigenvectors associated with the zero eigenvalue, respectively. The right eigenvector satisfies
\begin{equation}
J_cv=0.
\end{equation}
Choosing $v_2=1$, we obtain
\begin{equation}
v=
\begin{pmatrix}
-\dfrac{Q}{A}\\[4pt]
1\\[4pt]
\dfrac{\nu_2A-\nu_1Q}{\mu A}
\end{pmatrix}.
\label{eq:right_eigenvector}
\end{equation}
The left eigenvector satisfies
\begin{equation}
J_c^{T}w=0.
\end{equation}
A convenient normalized choice is
\begin{equation}
w=
\begin{pmatrix}
0\\
1\\
0
\end{pmatrix},
\qquad
w^{T}v=1.
\label{eq:left_eigenvector}
\end{equation}

The nonlinear component responsible for the bifurcation is the bilinear incidence term $\beta SI$. Consequently, the only second-order derivatives required in the center-manifold calculation are
\begin{equation}
\frac{\partial^2 f_2}{\partial S\,\partial I}
=
\frac{\partial^2 f_2}{\partial I\,\partial S}
=\beta.
\end{equation}
At the bifurcation point, $\beta=\beta_c$. Following the center-manifold formulation of Castillo--Chavez and Song, the first bifurcation coefficient is
\begin{equation}
a=
\frac{1}{2}
\sum_{k,i,j}
w_kv_iv_j
\frac{\partial^2 f_k}{\partial x_i\partial x_j}
(E_0,\beta_c).
\end{equation}
Because the second component is the only component contributing to this expression and $w_2=1$, this reduces to
\begin{equation}
a=\beta_cv_1v_2.
\end{equation}
Using $v_1=-Q/A$ and $v_2=1$, we obtain
\begin{equation}
a=-\frac{\beta_cQ}{A}.
\end{equation}
Substitution of $\beta_c=AQ/\lambda$ gives
\begin{equation}
a=-\frac{Q^2}{\lambda}
=-\frac{(\sigma+\nu_2+\mu)^2}{\lambda}<0.
\label{eq:bifurcation_a}
\end{equation}

The second bifurcation coefficient is
\begin{equation}
b=
\sum_{k,i}
w_kv_i
\frac{\partial^2 f_k}{\partial x_i\partial\beta}
(E_0,\beta_c).
\end{equation}
Since
\begin{equation}
\frac{\partial f_2}{\partial\beta}=SI,
\end{equation}
we have
\begin{equation}
\frac{\partial^2f_2}{\partial I\,\partial\beta}=S.
\end{equation}
At the disease-free equilibrium,
\begin{equation}
S_0=\frac{\lambda}{A},
\end{equation}
and therefore
\begin{equation}
b=\frac{\lambda}{A}
=\frac{\lambda}{\nu_1+\mu}>0.
\label{eq:bifurcation_b}
\end{equation}

Thus, the nondegeneracy conditions required for the center-manifold bifurcation theorem are satisfied:
\begin{equation}
a<0,\qquad b>0.
\end{equation}

\begin{theorem}
\label{thm:forward_bifurcation}
The disease-free equilibrium of the model undergoes a forward transcritical bifurcation at $\mathcal{R}_0=1$, equivalently at the critical transmission rate
\begin{equation}
\beta_c=\frac{(\nu_1+\mu)(\sigma+\nu_2+\mu)}{\lambda}.
\end{equation}
\end{theorem}

\begin{proof}
At $\beta=\beta_c$, the Jacobian evaluated at the disease-free equilibrium possesses exactly one simple zero eigenvalue, while its remaining eigenvalues are negative. Furthermore, the center-manifold coefficients satisfy $a<0$ and $b>0$. Hence, the nondegeneracy and transversality conditions of the center-manifold bifurcation theorem are fulfilled, and the model undergoes a forward transcritical bifurcation at $\mathcal{R}_0=1$.
\end{proof}

\subsection{Local normal form and direction of bifurcation}
\label{subsec_normal_form}

The local dynamics on the one-dimensional center manifold can be represented by the normal form
\begin{equation}
\dot z
=
az^2+b(\beta-\beta_c)z
+
O\left(|z|^3+|\beta-\beta_c|z^2\right).
\label{eq:normal_form_general}
\end{equation}
Using the coefficients obtained above, this becomes
\begin{equation}
\dot z
=
-\frac{Q^2}{\lambda}z^2
+
\frac{\lambda}{A}(\beta-\beta_c)z
+
O\left(|z|^3+|\beta-\beta_c|z^2\right).
\label{eq:normal_form}
\end{equation}
The nonzero equilibrium branch of the reduced system satisfies
\begin{equation}
z^*=-\frac{b}{a}(\beta-\beta_c).
\end{equation}
Since
\begin{equation}
-\frac{b}{a}>0,
\end{equation}
the biologically relevant nonzero branch exists for
\begin{equation}
\beta>\beta_c,
\end{equation}
which is equivalent to
\begin{equation}
\mathcal{R}_0>1.
\end{equation}
Consequently, the endemic equilibrium emerges from the disease-free equilibrium on the supercritical side of the threshold. This confirms that the bifurcation is forward transcritical rather than backward, and establishes analytically that endemic infection becomes feasible only after the transmission parameter crosses the critical value $\beta_c$.

\section{Two-parameter threshold geometry}
\label{sec_two_parameter}
The threshold analysis is extended to the two-parameter plane $(\beta,\nu_1)$, where $\beta$ denotes the transmission rate and $\nu_1$ represents the vaccination rate.

From the expression for the basic reproduction number,
\begin{equation}
\mathcal{R}_0(\beta,\nu_1)=\frac{\beta\lambda}{(\nu_1+\mu)(\sigma+\nu_2+\mu)},
\label{eq_R0_two_parameter}
\end{equation}
the epidemic threshold is determined by $\mathcal{R}_0(\beta,\nu_1)=1$. Solving this condition for $\beta$ gives the critical transmission threshold
\begin{equation}
\boxed{
\beta_c(\nu_1)=
\frac{(\nu_1+\mu)(\sigma+\nu_2+\mu)}{\lambda}.
}
\label{eq:beta_c_nu1}
\end{equation}
Thus, unlike the one-parameter setting in which $\beta_c$ is a single critical value, the threshold becomes a curve in the $(\beta,\nu_1)$ parameter plane.

Since $\beta_c(\nu_1)$ is linear in $\nu_1$, its derivative is
\begin{equation}
\frac{d\beta_c}{d\nu_1}
=
\frac{\sigma+\nu_2+\mu}{\lambda}>0.
\label{eq:threshold_slope}
\end{equation}
Therefore, increasing the vaccination rate shifts the critical transmission threshold toward larger values. Equivalently, a higher level of vaccination allows the population to tolerate a larger transmission rate before the epidemic threshold is crossed.

The same threshold relation can be solved for the critical vaccination rate. For a prescribed transmission rate $\beta$, we obtain
\begin{equation}
\nu_{1,c}(\beta)
=
\frac{\beta\lambda}{\sigma+\nu_2+\mu}-\mu.
\label{eq:nu1_critical}
\end{equation}
Consequently,
\begin{equation}
\nu_1>\nu_{1,c}(\beta)
\quad\Longrightarrow\quad
\mathcal{R}_0<1,
\end{equation}
whereas
\begin{equation}
\nu_1<\nu_{1,c}(\beta)
\quad\Longrightarrow\quad
\mathcal{R}_0>1.
\end{equation}
Hence, equation~\eqref{eq:nu1_critical} provides an explicit vaccination threshold required to maintain the system in the disease-free regime for a specified transmission intensity.

The threshold curve
\begin{equation}
\mathcal{B}
=
\left\{
(\beta,\nu_1):
\beta=
\frac{(\nu_1+\mu)(\sigma+\nu_2+\mu)}{\lambda}
\right\}
\label{eq:threshold_curve}
\end{equation}
divides the positive $(\beta,\nu_1)$ parameter plane into two epidemiologically distinct regions. Below the curve,
\begin{equation}
\beta<\beta_c(\nu_1)
\quad\Longleftrightarrow\quad
\mathcal{R}_0<1,
\end{equation}
and the disease-free equilibrium is locally asymptotically stable. Above the curve,
\begin{equation}
\beta>\beta_c(\nu_1)
\quad\Longleftrightarrow\quad
\mathcal{R}_0>1,
\end{equation}
and the endemic equilibrium exists and is locally asymptotically stable. Along $\mathcal{B}$, $\mathcal{R}_0=1$, the disease-free and endemic equilibria coincide, and the forward transcritical bifurcation occurs in the direction transverse to the threshold curve. Thus, the one-parameter bifurcation obtained by varying $\beta$ while fixing $\nu_1$ can be interpreted as a cross-section of the more general two-parameter threshold geometry.

The regularity of the threshold curve follows directly from the implicit function theorem. Define
\begin{equation}
G(\beta,\nu_1)=\mathcal{R}_0(\beta,\nu_1)-1
=
\frac{\beta\lambda}{(\nu_1+\mu)(\sigma+\nu_2+\mu)}-1.
\end{equation}
Its gradient is
\begin{equation}
\nabla G
=
\left(
\frac{\lambda}{(\nu_1+\mu)(\sigma+\nu_2+\mu)},
-\frac{\beta\lambda}{(\nu_1+\mu)^2(\sigma+\nu_2+\mu)}
\right).
\end{equation}
On the threshold curve $\mathcal{B}$, where $\beta\lambda=(\nu_1+\mu)(\sigma+\nu_2+\mu)$, this becomes
\begin{equation}
\nabla G\big|_{\mathcal{B}}
=
\left(
\frac{\lambda}{(\nu_1+\mu)(\sigma+\nu_2+\mu)},
-\frac{1}{\nu_1+\mu}
\right)\neq(0,0).
\end{equation}

Therefore, $\mathcal{B}$ is a regular smooth curve in the $(\beta,\nu_1)$ plane. The two-parameter analysis should consequently be interpreted as a one-dimensional transcritical bifurcation curve embedded in a two-dimensional parameter space, rather than as a codimension-two singular bifurcation.

The two-parameter threshold also provides a direct intervention interpretation. From
\begin{equation}
\beta\lambda=(\nu_1+\mu)(\sigma+\nu_2+\mu)
\label{eq:threshold_balance}
\end{equation}
we see that transmission is balanced by vaccination, treatment, and disease-induced removal. Moreover, the critical transmission threshold increases with the treatment rate and the disease-induced removal rate according to
\begin{equation}
\frac{\partial\beta_c}{\partial\nu_2}
=
\frac{\nu_1+\mu}{\lambda}>0,
\qquad
\frac{\partial\beta_c}{\partial\sigma}
=
\frac{\nu_1+\mu}{\lambda}>0.
\end{equation}
Thus, stronger treatment or greater disease-induced removal shifts the threshold toward larger transmission rates and makes epidemic persistence more difficult.

The two-parameter threshold analysis also determines how the endemic infection level varies within the endemic region. From the endemic equilibrium,
\begin{equation}
I^*
=
\frac{\lambda}{\sigma+\nu_2+\mu}
-
\frac{\nu_1+\mu}{\beta},
\label{eq:Istar_two_parameter}
\end{equation}
we obtain
\begin{equation}
\frac{\partial I^*}{\partial\beta}
=
\frac{\nu_1+\mu}{\beta^2}>0,
\label{eq:dI_dbeta}
\end{equation}
and
\begin{equation}
\frac{\partial I^*}{\partial\nu_1}
=
-\frac{1}{\beta}<0.
\label{eq:dI_dnu1}
\end{equation}
Hence, increasing the transmission rate increases the endemic infection level, whereas increasing vaccination decreases it. These relations show that vaccination has two complementary effects: it shifts the epidemic threshold toward larger transmission rates and simultaneously reduces the endemic infection burden once persistence has been established.

\begin{theorem}
\label{thm_two_parameter}
For fixed positive $\lambda$, $\mu$, $\sigma$, and $\nu_2$, the $(\beta,\nu_1)$ parameter plane is separated by the regular threshold curve
\begin{equation}
\mathcal{B}:\qquad
\beta=
\frac{(\nu_1+\mu)(\sigma+\nu_2+\mu)}{\lambda}.
\end{equation}
In the region below $\mathcal{B}$, $\mathcal{R}_0<1$ and the disease-free equilibrium is locally asymptotically stable. In the region above $\mathcal{B}$, $\mathcal{R}_0>1$, the endemic equilibrium exists and is locally asymptotically stable. Along $\mathcal{B}$, the disease-free and endemic equilibria coincide, and the system undergoes a forward transcritical bifurcation in the direction transverse to the threshold curve.
\end{theorem}

\begin{proof}
The threshold curve follows from $\mathcal{R}_0(\beta,\nu_1)=1$. The local stability properties of the two parameter regions follow from the stability results established previously, while the center-manifold analysis establishes the forward direction of the transcritical bifurcation. The nonvanishing gradient of $G(\beta,\nu_1)=\mathcal{R}_0-1$ ensures that the threshold set is a regular curve. Therefore, the curve $\mathcal{B}$ continuously separates the disease-free and endemic parameter regimes.
\end{proof}

\section{Analytical threshold structure}
\label{sec_threshold}

The results obtained above can be summarized through the following threshold structure:

\begin{equation}
\boxed{
\begin{array}{ccl}
\Rzero<1
&
\Longrightarrow
&
E_0\text{ locally asymptotically stable},
\\[2mm]
\Rzero=1
&
\Longrightarrow
&
E_0=E^*\text{ and forward transcritical bifurcation},
\\[2mm]
\Rzero>1
&
\Longrightarrow
&
E^*\text{ exists and is locally asymptotically stable}.
\end{array}
}
\label{eq_thresholdsummary}
\end{equation}

The critical transmission threshold is
\begin{equation}
\boxed{
\beta_c
=
\frac{(\nu_1+\mu)(\sigma+\nu_2+\mu)}
{\lambda}.
}
\end{equation}

Consequently,
\begin{equation}
\beta<\beta_c
\quad\Longleftrightarrow\quad
\Rzero<1,
\qquad
\beta>\beta_c
\quad\Longleftrightarrow\quad
\Rzero>1.
\end{equation}

The explicit endemic branch can be written in terms of the distance from the critical transmission rate as
\begin{equation}
I^*
=
\frac{\lambda}{Q\beta}(\beta-\beta_c).
\label{eq_Istar_betac}
\end{equation}

Hence,
\begin{equation}
I^*<0 \quad\text{for}\quad \beta<\beta_c,
\qquad
I^*=0 \quad\text{for}\quad \beta=\beta_c,
\qquad
I^*>0 \quad\text{for}\quad \beta>\beta_c.
\end{equation}

Moreover,
\begin{equation}
\left.
\frac{\partial I^*}{\partial\beta}
\right|_{\beta=\beta_c}
=
\frac{\nu_1+\mu}{\beta_c^2}>0.
\end{equation}

Thus, the endemic equilibrium branch crosses the disease-free equilibrium at $\beta=\beta_c$ and enters the biologically feasible region as $\beta$ increases through $\beta_c$, confirming the forward direction predicted by the center-manifold analysis.

\section{Numerical results and bifurcation analysis}
\label{sec_numerical}

This section presents numerical simulations to complement the analytical results established in the preceding sections. The numerical experiments illustrate the long-term behavior of the model in the disease-free and endemic regimes, verify the threshold behavior associated with the basic reproduction number, confirm the forward transcritical bifurcation through numerical equilibrium continuation and eigenvalue analysis, and investigate the two-parameter interaction between transmission and vaccination. Finally, local and global sensitivity analyses are presented to identify the parameters that most strongly influence the basic reproduction number and the endemic infected population.

Unless otherwise stated, the numerical simulations use the parameter set
\begin{equation}
\lambda=0.0121,\qquad \beta=0.95,\qquad \mu=0.0069,\qquad
\sigma=0.0024,\qquad \nu_1=0.90,\qquad \nu_2=0.05.
\label{eq:numerical_parameters}
\end{equation}
The initial population is taken as
\begin{equation}
(S(0),I(0),R(0))=(0.7,0.3,0).
\label{eq:numerical_initial}
\end{equation}
Numerical solutions are obtained using the variable-step fourth-order Runge-Kutta method implemented through the MATLAB \texttt{ode45} solver, with relative and absolute tolerances of $10^{-9}$ and $10^{-11}$, respectively. The parameter values used in the numerical simulations are summarized in Table~\ref{tab_parameters}.

\begin{table}[htbp]
\centering
\begin{tabular}{cccc}
\hline
Parameter & Value & Description & Unit\\
\hline
$\lambda$ & $0.0121$ & Recruitment rate & year$^{-1}$\\
$\beta$ & $0.95$ or $5$ & Transmission rate & year$^{-1}$\\
$\mu$ & $0.0069$ & Natural mortality rate & year$^{-1}$\\
$\sigma$ & $0.0024$ & Disease-induced mortality rate & year$^{-1}$\\
$\nu_1$ & $0.90$ & Vaccination rate & year$^{-1}$\\
$\nu_2$ & $0.05$ & Treatment rate & year$^{-1}$\\
\hline
\end{tabular}
\caption{Parameter values used in the numerical simulations.}
\label{tab_parameters}
\end{table}

\subsection{Numerical dynamics in the disease-free and endemic regimes}
\label{subsec:numerical_dynamics}

We first consider the baseline transmission and vaccination parameters $\beta=0.95$ and $\nu_1=0.90$. For this parameter combination,
\begin{equation}
\mathcal{R}_0=\frac{(0.95)(0.0121)}{(0.90+0.0069)(0.0593)}
\approx0.213744<1.
\end{equation}
Hence, the analytical stability results predict that the disease-free equilibrium is locally asymptotically stable. The corresponding equilibrium is
\begin{equation}
E_0=
\left(
\frac{\lambda}{\nu_1+\mu},
0,
\frac{\lambda\nu_1}{\mu(\nu_1+\mu)}
\right)
\approx(0.013342,0,1.740281).
\end{equation}
Although the initial condition contains infected individuals, the numerical solution shows that the infected population decreases toward zero while the susceptible and recovered populations approach their corresponding disease-free equilibrium values. This behavior agrees with the analytical stability result for $\mathcal{R}_0<1$.

To illustrate the endemic regime, we next increase the transmission parameter to $\beta=5$ while keeping $\nu_1=0.90$ and the remaining parameters fixed. For this parameter combination,
\begin{equation}
\mathcal{R}_0=\frac{(5)(0.0121)}{(0.9069)(0.0593)}
\approx1.124971>1.
\end{equation}
Consequently, the disease-free equilibrium loses stability and a biologically feasible endemic equilibrium emerges. The endemic equilibrium is
\begin{equation}
E^*=(S^*,I^*,R^*),
\end{equation}
where
\begin{equation}
S^*=\frac{\sigma+\nu_2+\mu}{\beta},
\qquad
I^*=\frac{\lambda}{\sigma+\nu_2+\mu}-\frac{\nu_1+\mu}{\beta},
\qquad
R^*=\frac{\nu_1S^*+\nu_2I^*}{\mu}.
\end{equation}
For $\beta=5$ and $\nu_1=0.90$, this gives
\begin{equation}
E^*\approx(0.011860,0.022667,1.711212).
\end{equation}
The numerical solution converges toward this positive endemic equilibrium, confirming persistence of infection when $\mathcal{R}_0>1$.

\begin{figure}[htbp]
\centering
\begin{subfigure}[b]{0.4\textwidth}
\centering
\includegraphics[width=\textwidth]{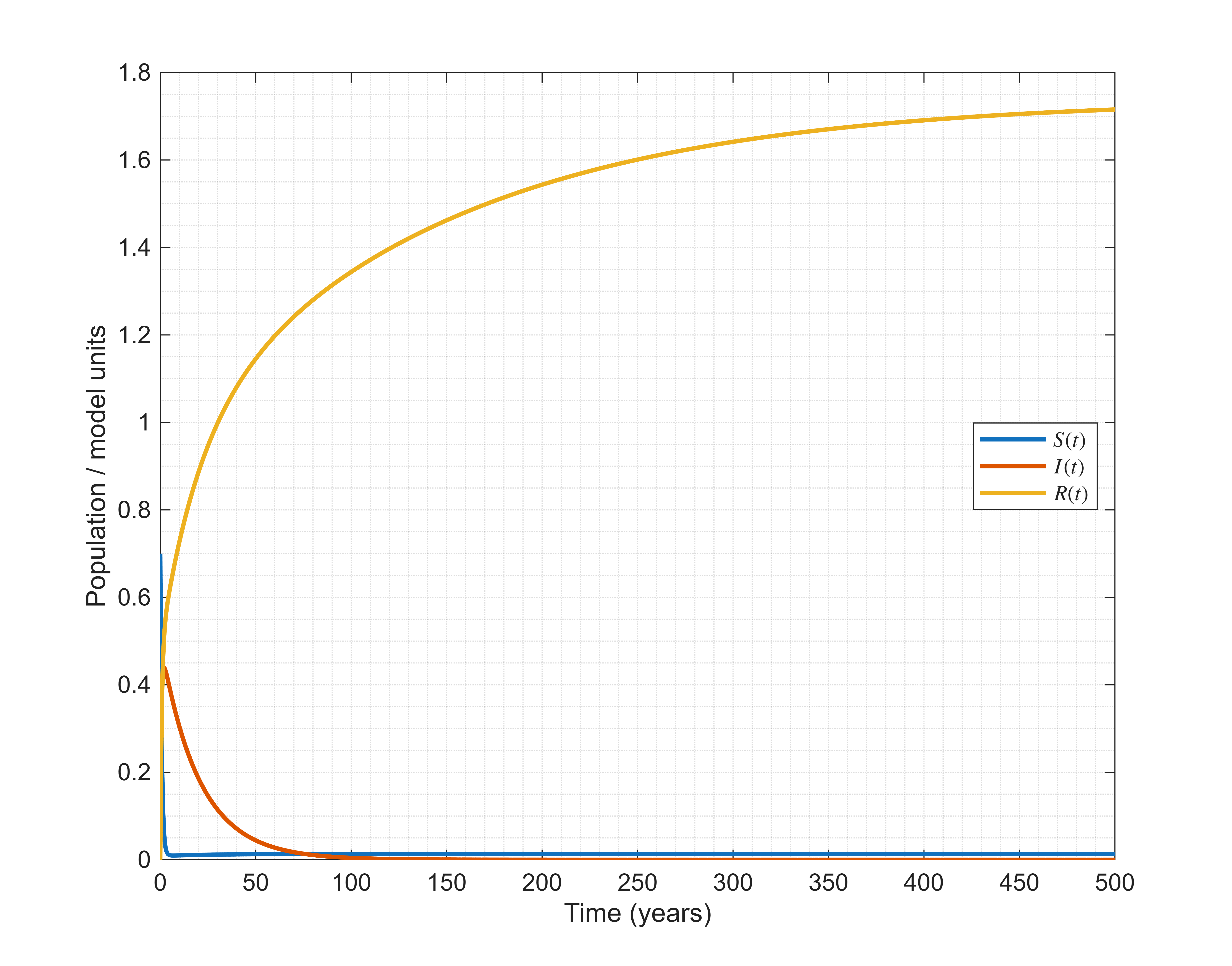}
\caption{Disease-free regime for $\mathcal{R}_0<1$.}
\label{fig:dfe_dynamics}
\end{subfigure}
\hfill
\begin{subfigure}[b]{0.4\textwidth}
\centering
\includegraphics[width=\textwidth]{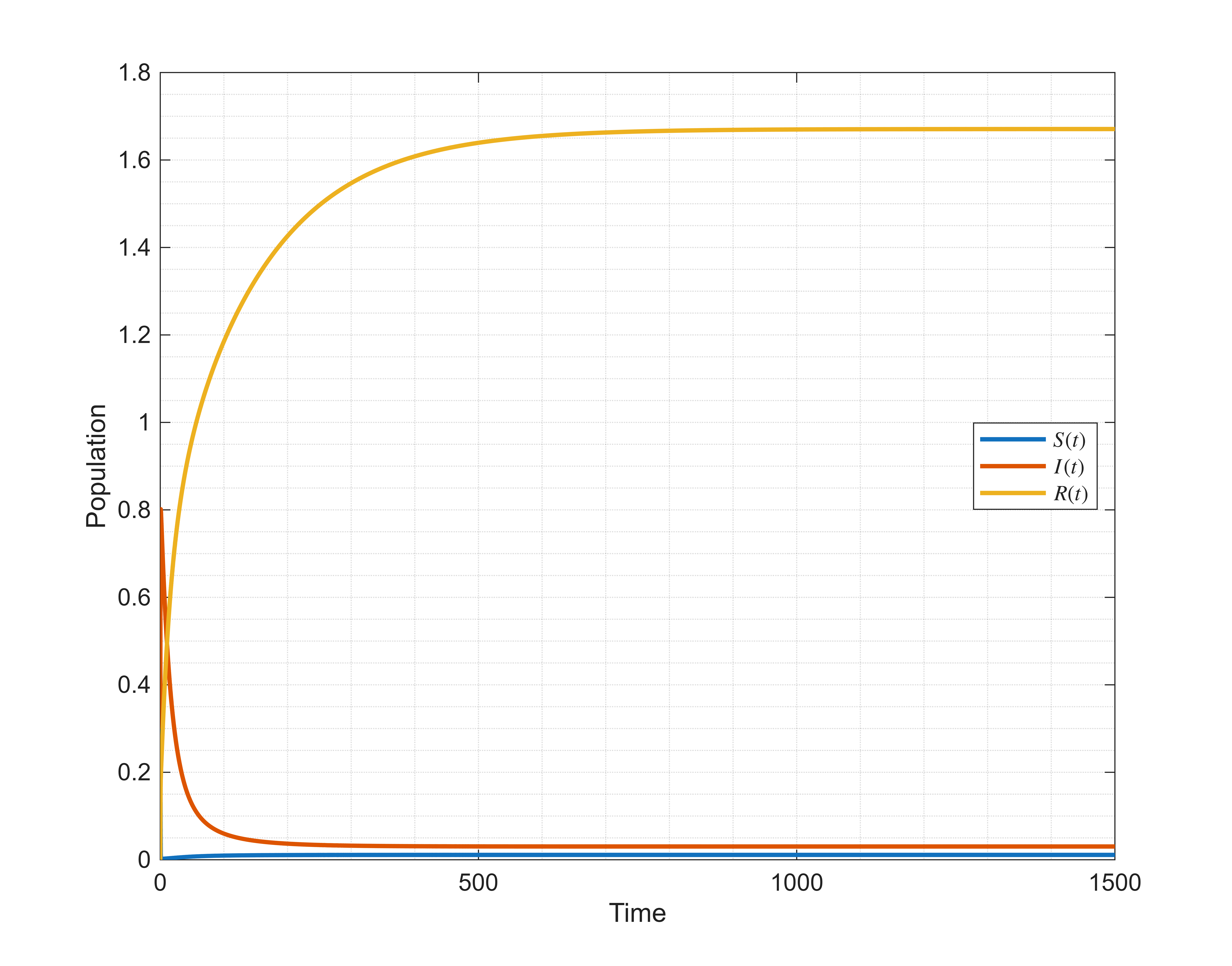}
\caption{Endemic regime for $\mathcal{R}_0>1$.}
\label{fig:endemic_dynamics}
\end{subfigure}
\caption{Numerical dynamics of the model under the two epidemic regimes. Panel (a) shows convergence toward the disease-free equilibrium when $\mathcal{R}_0<1$, whereas panel (b) shows convergence toward the endemic equilibrium when $\mathcal{R}_0>1$.}
\label{fig:dynamics}
\end{figure}

\subsection{Threshold behavior and one-parameter bifurcation}
\label{subsec_one_parameter_numerical}

The dependence of the basic reproduction number on the transmission parameter is first examined while all other parameters are fixed at their baseline values. Figure~\ref{fig_threshold_bifurcation} presents the resulting variation of $\mathcal{R}_0$ together with the corresponding equilibrium bifurcation diagram.

The reproduction number increases monotonically with $\beta$ and intersects the threshold $\mathcal{R}_0=1$ at the critical transmission rate
\begin{equation}
\beta_c\approx4.45.
\end{equation}
More precisely, for $\nu_1=0.90$,
\begin{equation}
\beta_c=4.444560.
\end{equation}
Therefore,
\begin{equation}
\beta<\beta_c\Longrightarrow\mathcal{R}_0<1,
\qquad
\beta=\beta_c\Longrightarrow\mathcal{R}_0=1,
\qquad
\beta>\beta_c\Longrightarrow\mathcal{R}_0>1.
\end{equation}
For the baseline transmission value $\beta=0.95$, the reproduction number is approximately $0.214$, placing the baseline parameter set in the disease-free regime.

The one-parameter equilibrium diagram confirms the analytical bifurcation result. The disease-free branch corresponds to $I^*=0$, while the endemic branch emerges from this branch at $\beta=\beta_c$. Below the critical value, the disease-free equilibrium is stable, and the endemic equilibrium is not biologically feasible. Above the critical value, a positive endemic branch emerges and becomes stable, while the disease-free equilibrium loses stability. Thus, the numerical equilibrium continuation confirms the forward transcritical bifurcation established analytically in Section~\ref{sec_bifurcation}.

\begin{figure}[htbp]
\centering
\begin{subfigure}{0.4\textwidth}
\centering
\includegraphics[width=\textwidth]{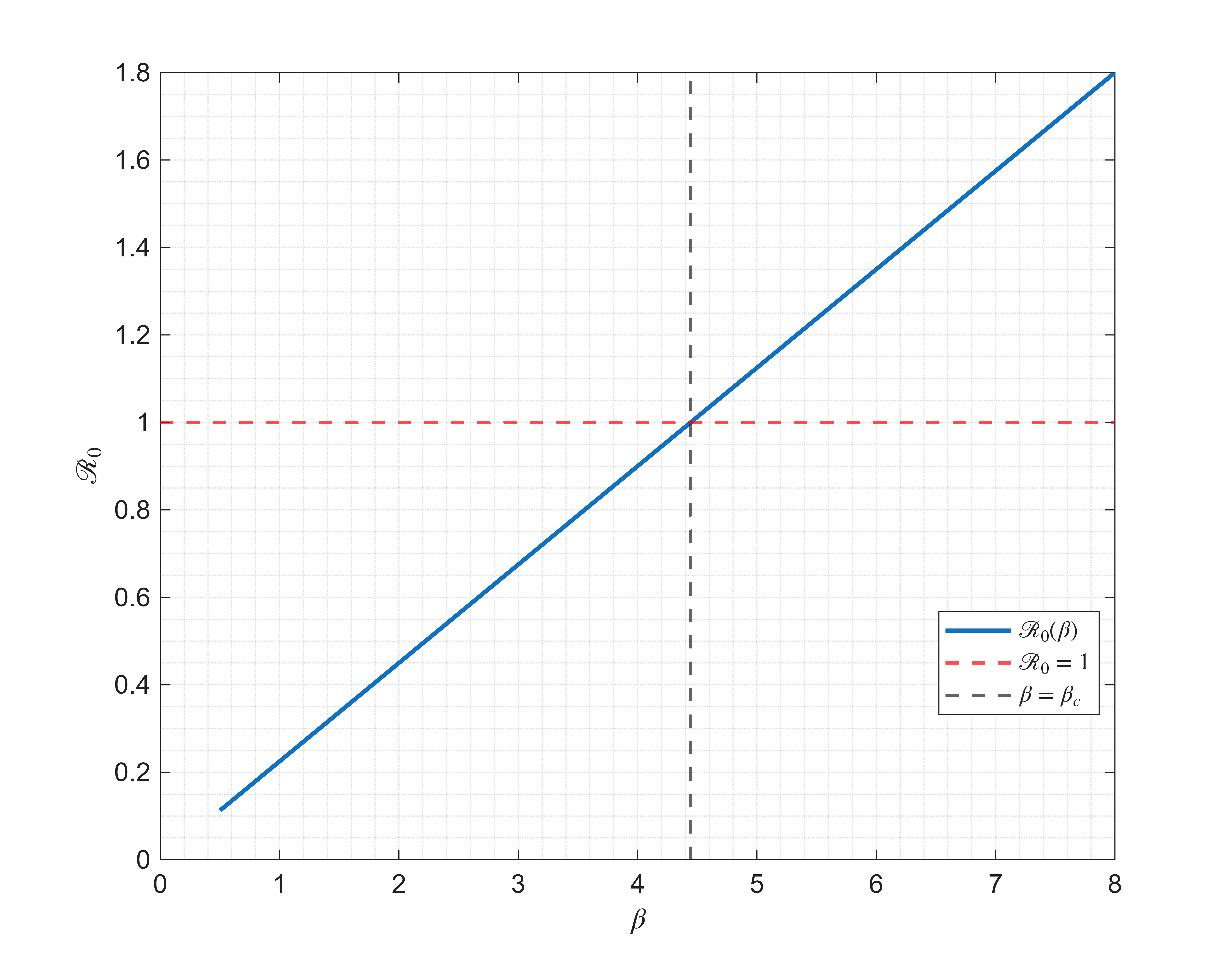}
\caption{Variation of $\mathcal{R}_0$ with $\beta$.}
\label{fig:R0_beta}
\end{subfigure}
\hfill
\begin{subfigure}{0.4\textwidth}
\centering
\includegraphics[width=\textwidth]{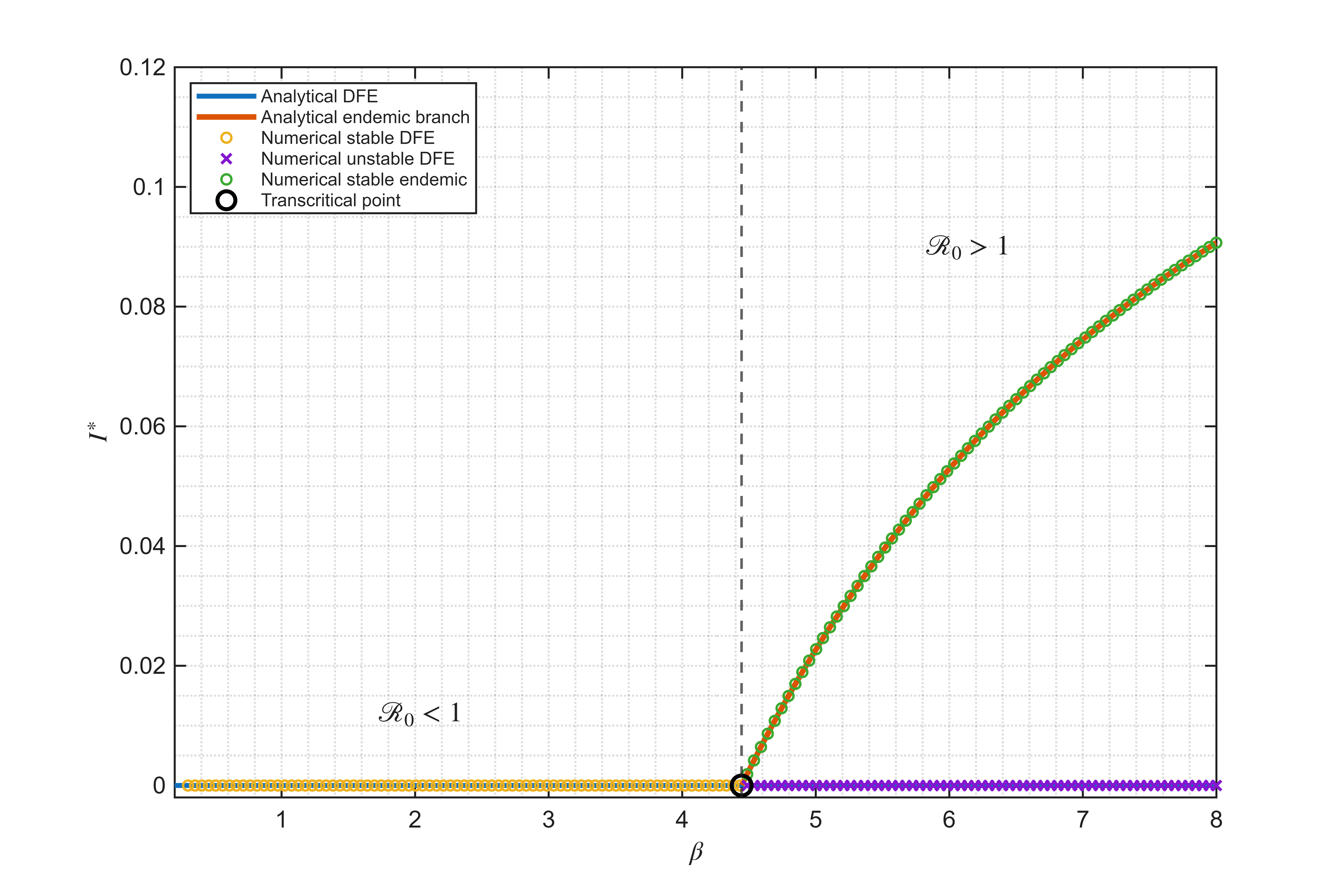}
\caption{One-parameter bifurcation diagram.}
\label{fig_bifurcation}
\end{subfigure}
\caption{Threshold and one-parameter bifurcation analysis with respect to the transmission parameter $\beta$. Panel (a) shows the variation of the basic reproduction number $\mathcal{R}_0$ and its intersection with the threshold $\mathcal{R}_0=1$. Panel (b) shows the equilibrium branches and the exchange of stability at the critical transmission value $\beta_c$.}
\label{fig_threshold_bifurcation}
\end{figure}

\subsection{Numerical verification through eigenvalue analysis}
\label{subsec_eigenvalue_verification}

To independently verify the stability exchange observed in the bifurcation diagram, the eigenvalues of the Jacobian evaluated at the relevant equilibrium points are computed over a range of transmission rates. Let $J(E,\beta)$ denote the Jacobian evaluated at an equilibrium $E$. Stability is determined by
\begin{equation}
\Lambda_{\max}(\beta)=\max_i\{\operatorname{Re}(\lambda_i(\beta))\}.
\end{equation}
An equilibrium is locally asymptotically stable whenever $\Lambda_{\max}<0$, whereas $\Lambda_{\max}>0$ indicates instability.

The maximum real part of the Jacobian eigenvalues remains negative below the critical transmission rate, approaches zero as $\beta$ approaches $\beta_c$, and crosses zero at the bifurcation point. Beyond the threshold, the corresponding value becomes positive for the disease-free equilibrium, confirming its instability when $\mathcal{R}_0>1$. Thus, the eigenvalue calculation provides an independent numerical verification of the stability exchange associated with the forward transcritical bifurcation.

\begin{figure}[htbp]
\centering
\includegraphics[width=0.68\textwidth]{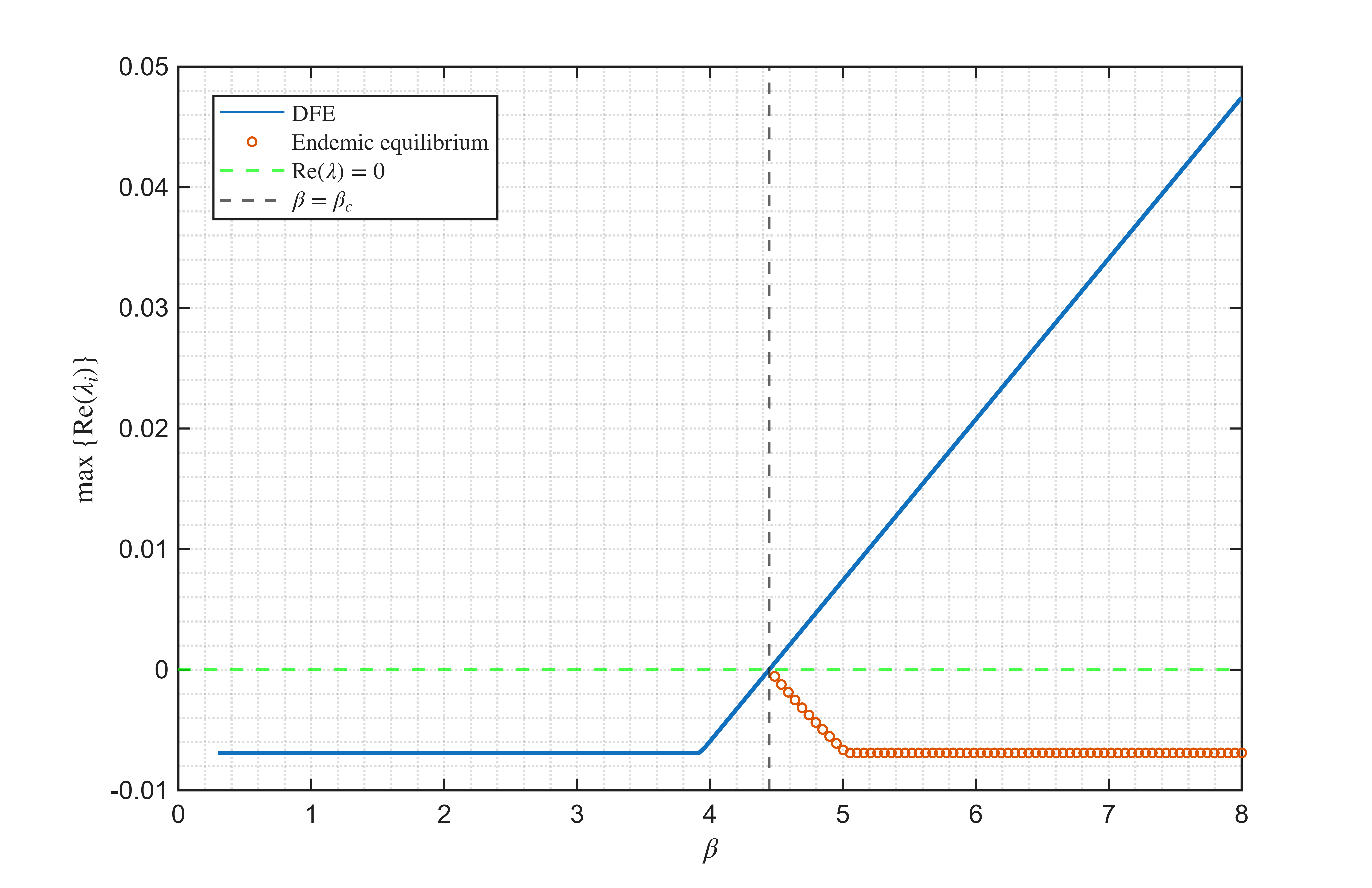}
\caption{Numerical stability verification of the disease-free equilibrium using the maximum real part of the Jacobian eigenvalues as a function of the transmission rate $\beta$. The zero crossing occurs at the critical value $\beta_c$, confirming the loss of stability of the disease-free equilibrium.}
\label{fig:eigenvalue}
\end{figure}

\subsection{Two-parameter threshold analysis}
\label{subsec_two_parameter_numerical}

The preceding analysis considers $\beta$ as the sole bifurcation parameter. We now investigate the interaction between transmission and vaccination by varying $(\beta,\nu_1)$ simultaneously. The threshold condition
\begin{equation}
\mathcal{R}_0(\beta,\nu_1)=1
\end{equation}
defines the boundary between the disease-free and endemic parameter regimes. The corresponding critical transmission rate is
\begin{equation}
\beta_c(\nu_1)=\frac{(\nu_1+\mu)(\sigma+\nu_2+\mu)}{\lambda}.
\end{equation}
The numerical representation confirms that $\beta_c$ increases monotonically with $\nu_1$. For the baseline vaccination parameter $\nu_1=0.90$, the critical transmission rate is approximately $\beta_c=4.45$. Thus, increasing vaccination shifts the epidemic threshold toward larger transmission rates.

The two-parameter threshold curve is shown together with the distribution of $\mathcal{R}_0(\beta,\nu_1)$ in Figure~\ref{fig_two_parameter_threshold}. Below the threshold curve, $\mathcal{R}_0<1$ and the disease-free equilibrium is the epidemiologically relevant stable state. Above the curve, $\mathcal{R}_0>1$ and the endemic equilibrium becomes feasible and stable. The heat-map representation provides a complementary view of the same threshold geometry and demonstrates that increasing transmission moves the system toward the endemic region, whereas increasing vaccination shifts it toward the disease-free region.

\begin{figure}[htbp]
\centering
\begin{subfigure}[b]{0.32\textwidth}
\centering
\includegraphics[width=\textwidth]{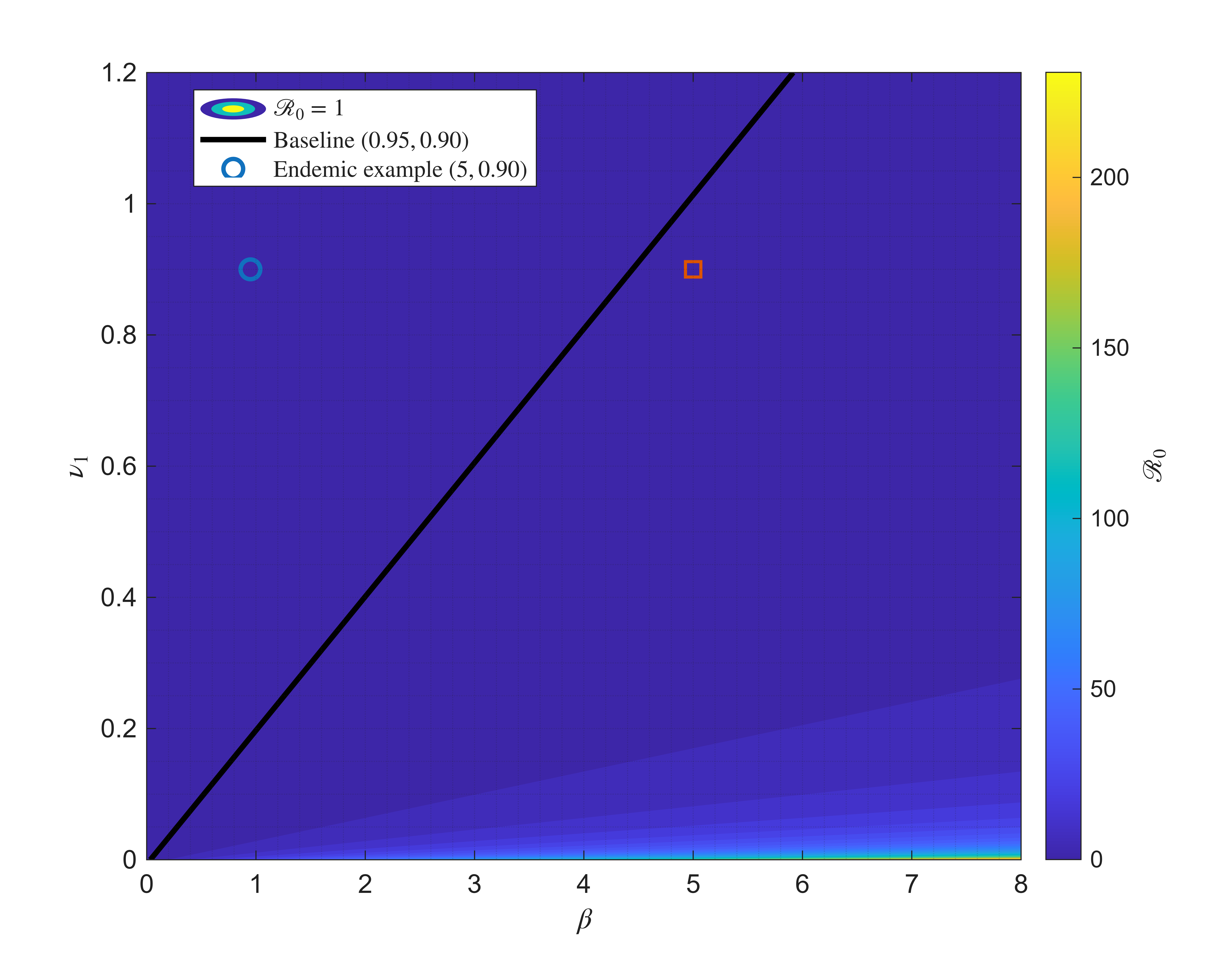}
\caption{$\beta_c$ versus $\nu_1$.}
\label{fig_beta_c_nu1}
\end{subfigure}
\hfill
\begin{subfigure}[b]{0.32\textwidth}
\centering
\includegraphics[width=\textwidth]{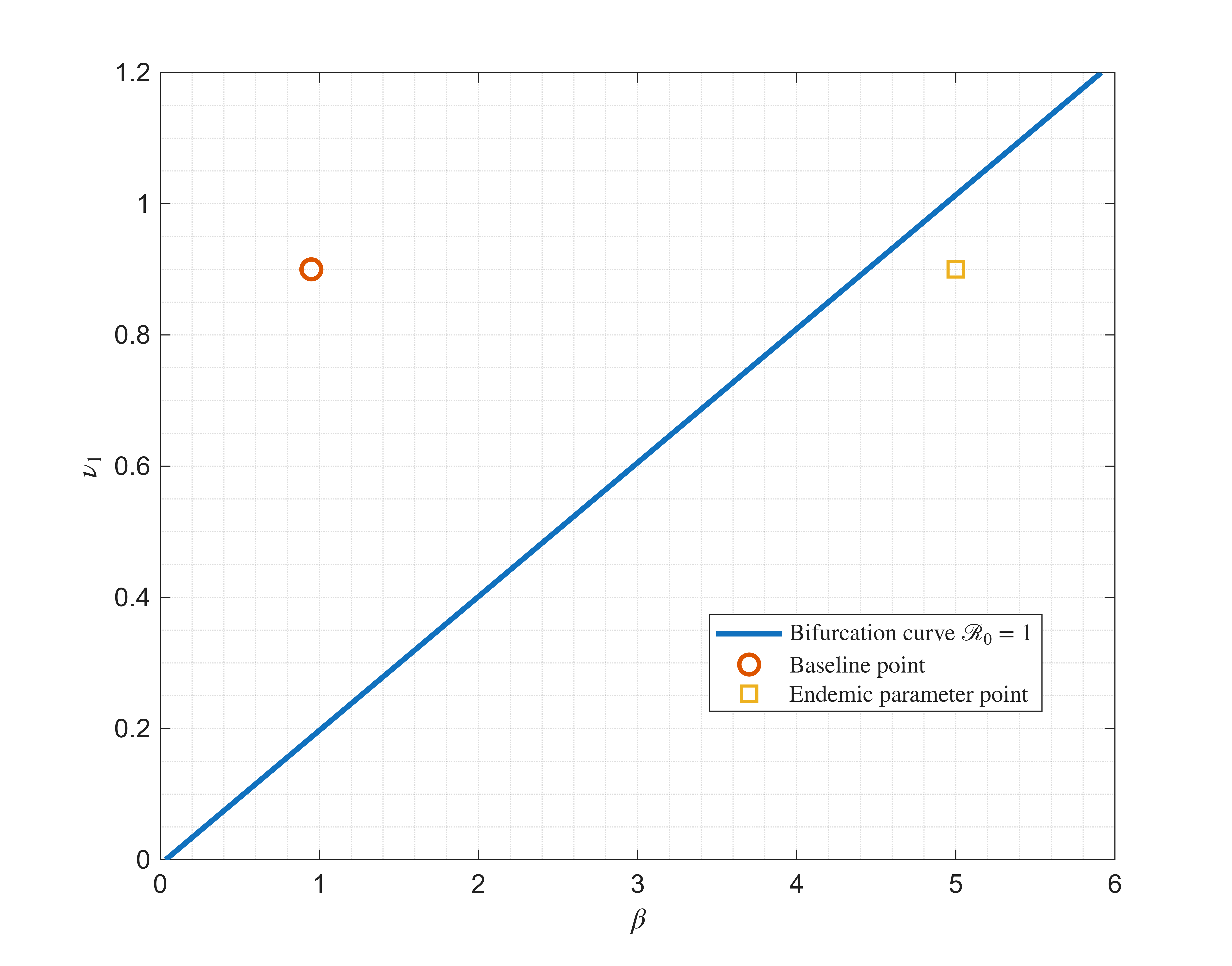}
\caption{Threshold curve $\mathcal{R}_0=1$.}
\label{fig_threshold_curve}
\end{subfigure}
\hfill
\begin{subfigure}[b]{0.32\textwidth}
\centering
\includegraphics[width=\textwidth]{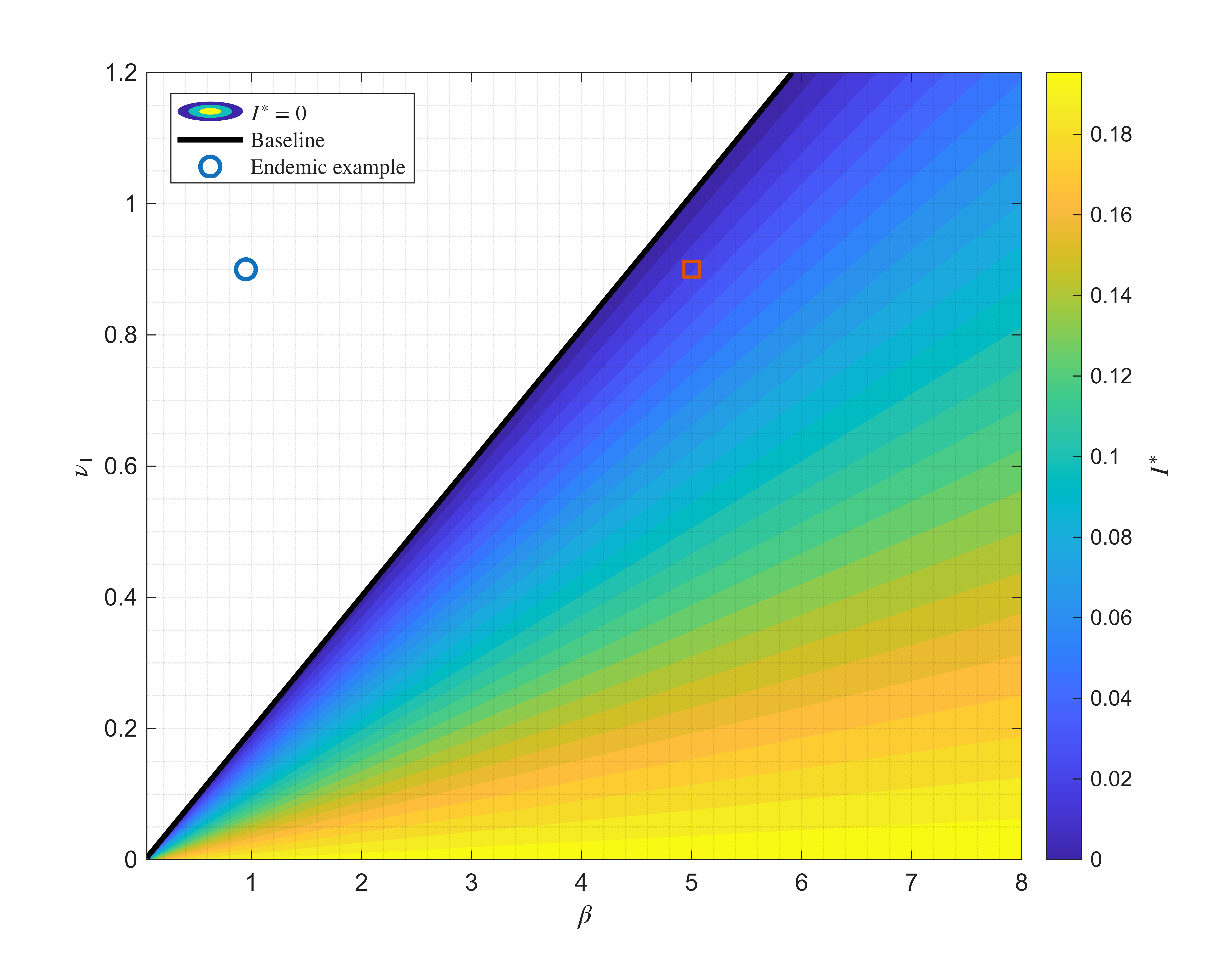}
\caption{Heat map of $\mathcal{R}_0(\beta,\nu_1)$.}
\label{fig_R0_heatmap}
\end{subfigure}
\caption{Two-parameter threshold geometry in the $(\beta,\nu_1)$ parameter plane. Panel (a) illustrates the dependence of the critical transmission rate $\beta_c$ on vaccination. Panel (b) shows the threshold curve $\mathcal{R}_0(\beta,\nu_1)=1$, which separates the disease-free and endemic parameter regimes. Panel (c) presents the corresponding distribution of $\mathcal{R}_0$ over the parameter plane.}
\label{fig_two_parameter_threshold}
\end{figure}

\subsection{Two-parameter endemic infection analysis}
\label{subsec_endemic_surface_numerical}

To examine the magnitude of infection within the endemic region, the endemic infected population
\begin{equation}
I^*=I^*(\beta,\nu_1)
=
\frac{\lambda}{\sigma+\nu_2+\mu}
-\frac{\nu_1+\mu}{\beta}
\end{equation}
is evaluated over the two-parameter space. Figure~\ref{fig_endemic_surface} presents the resulting three-dimensional representation of the endemic infected population as a function of the transmission rate $\beta$ and vaccination rate $\nu_1$.

The three-dimensional surface illustrates the dependence of the endemic infection level on the transmission and vaccination parameters. For a fixed vaccination rate, increasing $\beta$ increases the endemic infection level. Conversely, for a fixed transmission rate, increasing $\nu_1$ decreases the endemic infection burden.

\begin{figure}[htbp]
\centering
\includegraphics[width=0.62\textwidth]{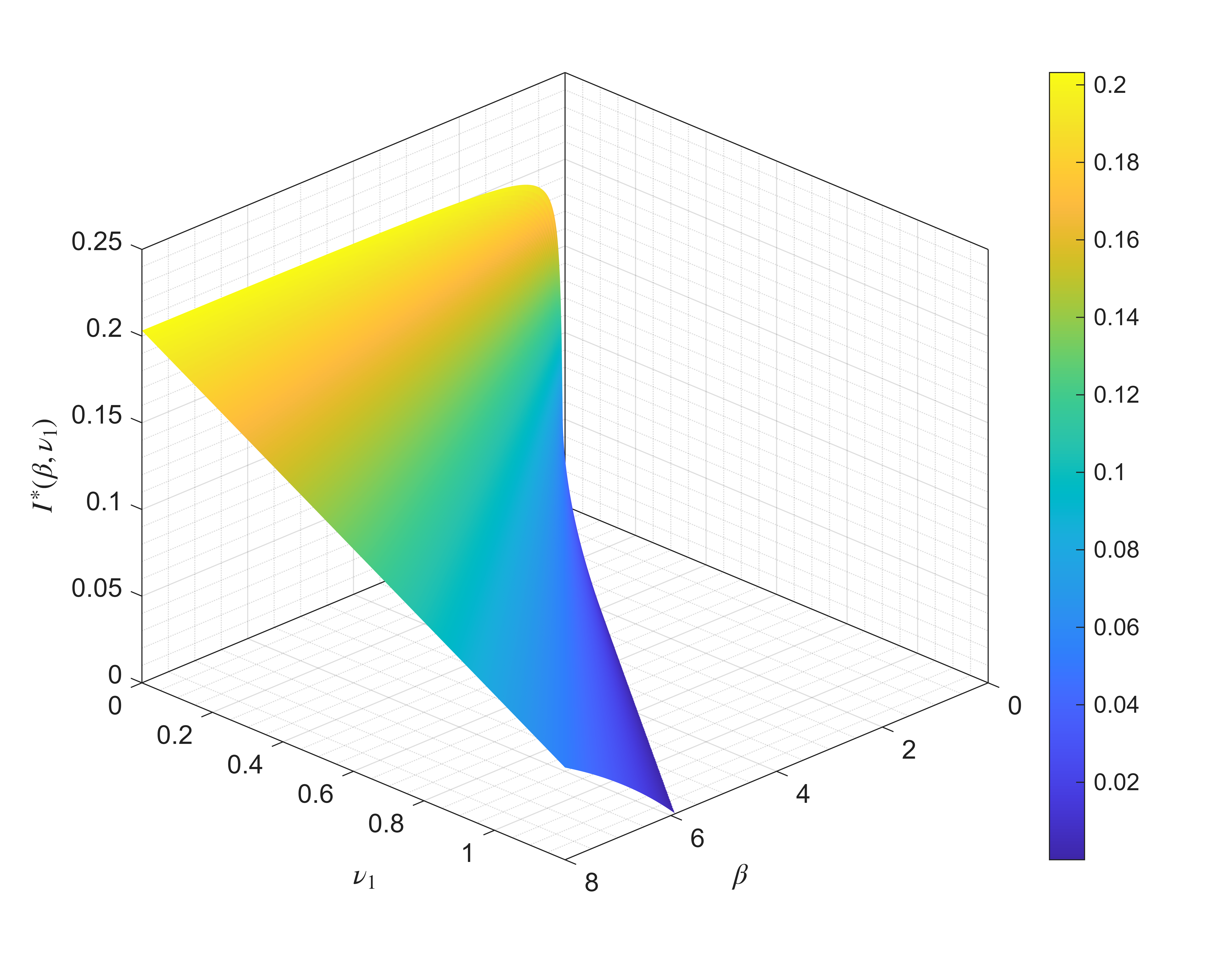}
\caption{Three-dimensional representation of the endemic infected population $I^*(\beta,\nu_1)$ over the transmission-vaccination parameter plane. The endemic infection level increases with the transmission rate $\beta$ and decreases with the vaccination rate $\nu_1$.}
\label{fig_endemic_surface}
\end{figure}

\subsection{Local sensitivity analysis of  \texorpdfstring{$\mathcal{R}_0=1$}{R0=1}}
\label{subsec_local_sensitivity}

The normalized sensitivity indices of the basic reproduction number are used to quantify the relative influence of the model parameters around the baseline parameter set. Since
\begin{equation}
\mathcal{R}_0=
\frac{\beta\lambda}{(\nu_1+\mu)(\sigma+\nu_2+\mu)},
\end{equation}
the normalized sensitivity indices are
\begin{equation}
\Upsilon_\beta^{\mathcal{R}_0}=1,
\qquad
\Upsilon_\lambda^{\mathcal{R}_0}=1,
\end{equation}
\begin{equation}
\Upsilon_{\nu_1}^{\mathcal{R}_0}
=
-\frac{\nu_1}{\nu_1+\mu},
\qquad
\Upsilon_{\sigma}^{\mathcal{R}_0}
=
-\frac{\sigma}{\sigma+\nu_2+\mu},
\end{equation}
\begin{equation}
\Upsilon_{\nu_2}^{\mathcal{R}_0}
=
-\frac{\nu_2}{\sigma+\nu_2+\mu},
\end{equation}
and
\begin{equation}
\Upsilon_{\mu}^{\mathcal{R}_0}
=
-\frac{\mu}{\nu_1+\mu}
-\frac{\mu}{\sigma+\nu_2+\mu}.
\end{equation}
The resulting sensitivity profile demonstrates the positive influence of transmission and recruitment on $\mathcal{R}_0$, while vaccination, treatment, and removal mechanisms act in the direction of reducing the reproduction number.

\begin{figure}[htbp]
\centering
\includegraphics[width=0.32\textwidth]{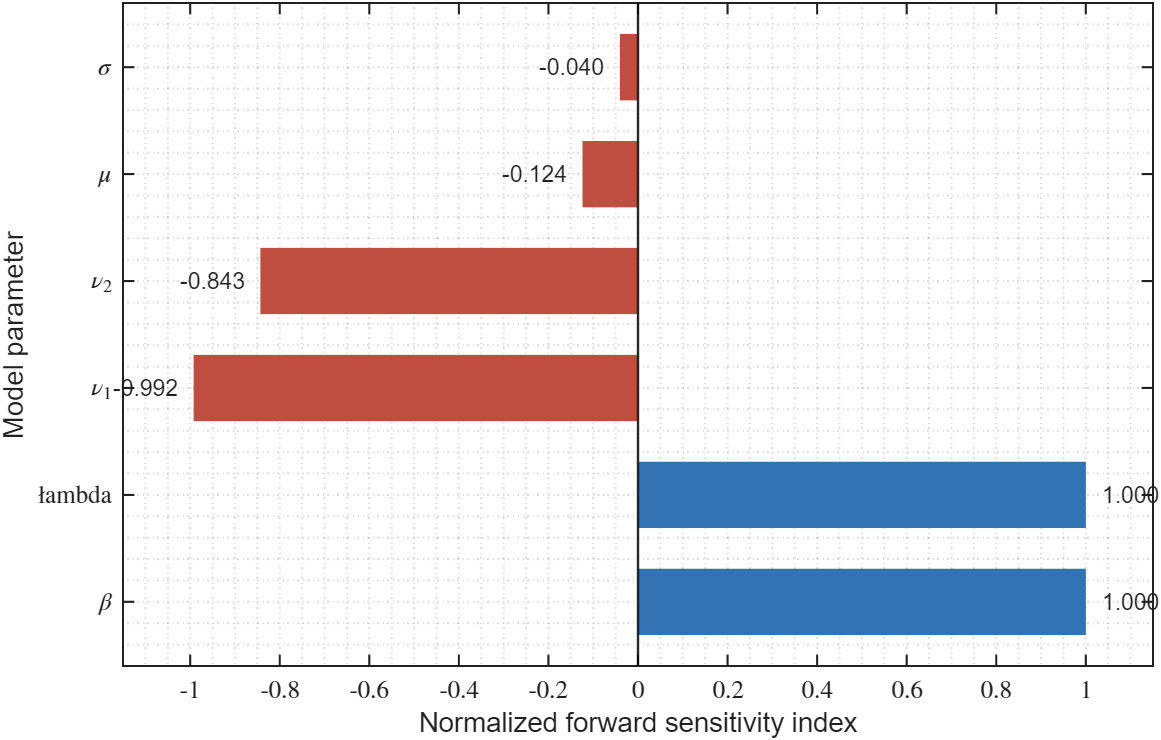}
\caption{Normalized sensitivity indices of the basic reproduction number $\mathcal{R}_0$ with respect to the model parameters. Positive indices indicate parameters that increase $\mathcal{R}_0$, whereas negative indices indicate parameters that suppress the epidemic threshold.}
\label{fig:sensitivity_r0}
\end{figure}

\subsection{Global LHS-PRCC analysis of the endemic infected population}
\label{subsec_prcc_I}

To complement the local sensitivity analysis, a global sensitivity analysis based on Latin hypercube sampling (LHS) and partial rank correlation coefficients (PRCCs) is performed for the endemic infected population $I^*$. The LHS procedure generates a representative ensemble of parameter sets from the prescribed parameter ranges, and the corresponding endemic equilibrium is evaluated for each realization.

The resulting distribution of $I^*$ is shown in Figure~\ref{fig:global_sensitivity_I}. The nominal endemic value for the baseline parameter set is approximately
\begin{equation}
I^*=0.0227.
\end{equation}
The sampled distribution demonstrates substantial variation in the endemic infection burden under parameter uncertainty. The corresponding PRCC analysis identifies the parameters responsible for this variation. The transmission rate $\beta$ exhibits a strong positive association with $I^*$, with
\begin{equation}
\mathrm{PRCC}_{\beta}\approx0.866,
\end{equation}
whereas vaccination and treatment exhibit strong negative associations, with approximately
\begin{equation}
\mathrm{PRCC}_{\nu_1}\approx-0.872,
\qquad
\mathrm{PRCC}_{\nu_2}\approx-0.886.
\end{equation}
The mortality-related parameters exhibit comparatively weaker associations.

\begin{figure}[htbp]
\centering
\begin{subfigure}[b]{0.4\textwidth}
\centering
\includegraphics[width=\textwidth]{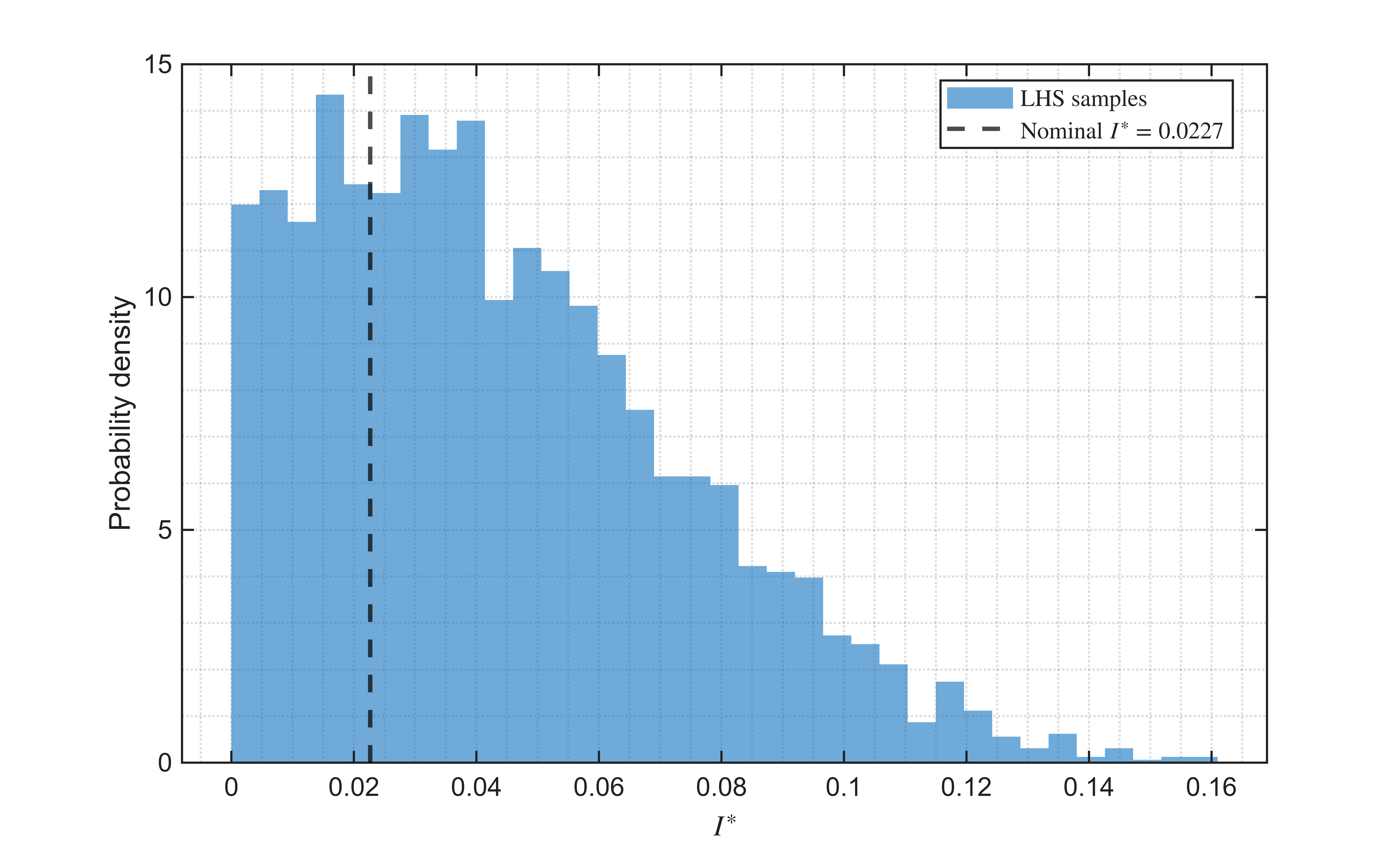}
\caption{LHS distribution of $I^*$.}
\label{fig:LHS_Istar}
\end{subfigure}
\hfill
\begin{subfigure}[b]{0.4\textwidth}
\centering
\includegraphics[width=\textwidth]{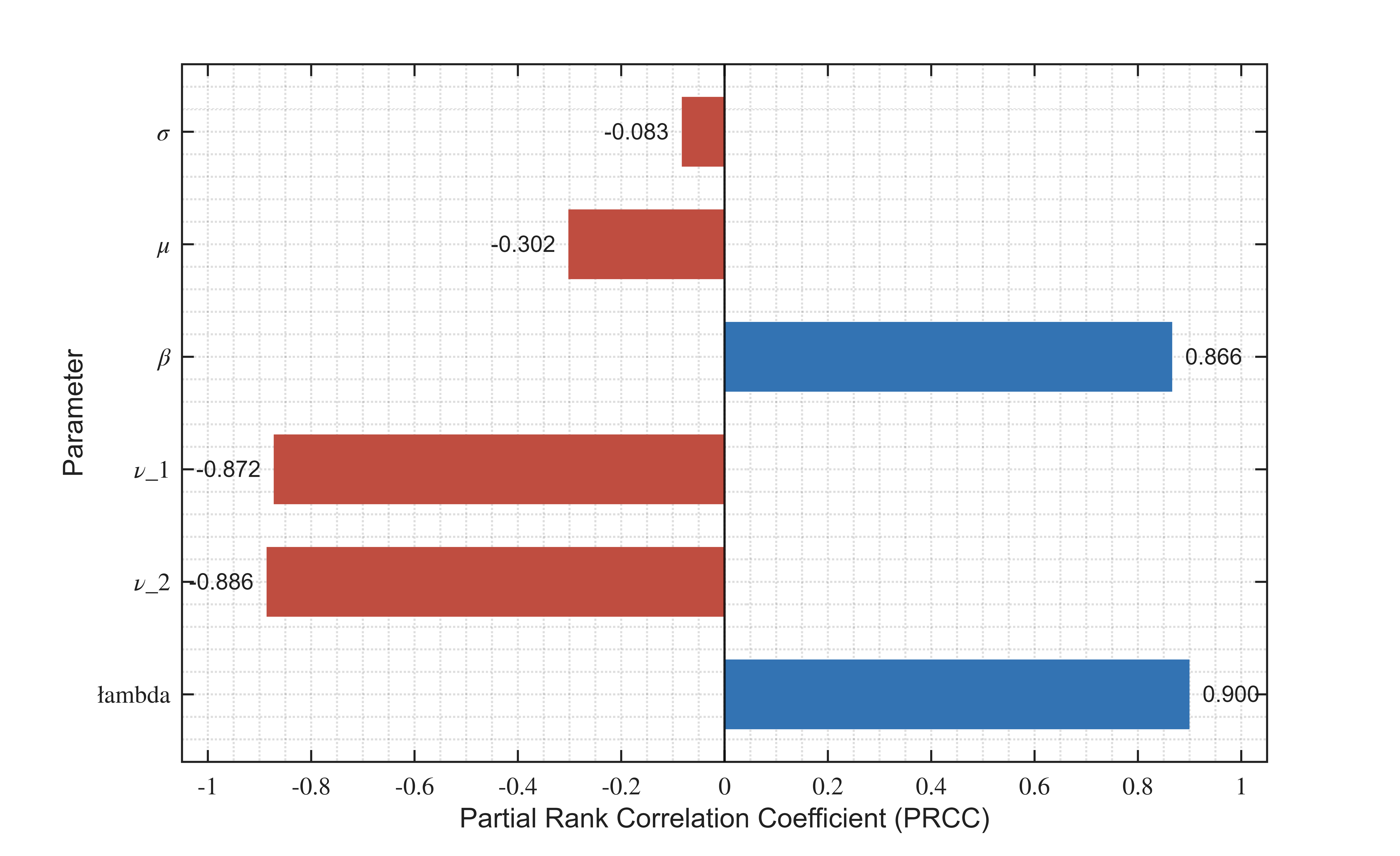}
\caption{PRCC analysis of $I^*$.}
\label{fig:PRCC_Istar}
\end{subfigure}
\caption{Global sensitivity analysis of the endemic infected population $I^*$. Panel (a) presents the distribution of $I^*$ obtained through Latin hypercube sampling, while panel (b) shows the corresponding partial rank correlation coefficients.}
\label{fig:global_sensitivity_I}
\end{figure}

\subsection{Global LHS-PRCC analysis of the basic reproduction number}
\label{subsec_prcc_R0}

A corresponding global sensitivity analysis is performed for the basic reproduction number. Figure~\ref{fig:global_sensitivity_R0} presents the LHS distribution of $\mathcal{R}_0$ and the associated PRCC results. For the baseline parameter set, the nominal reproduction number is approximately
\begin{equation}
\mathcal{R}_0=0.214.
\end{equation}
The LHS distribution demonstrates that parameter uncertainty can substantially modify the transmission potential of HBV and may move the system toward the threshold $\mathcal{R}_0=1$.

The PRCC analysis identifies the transmission rate and recruitment rate as the dominant positive contributors to $\mathcal{R}_0$, with
\begin{equation}
\mathrm{PRCC}_{\beta}\approx0.856,
\qquad
\mathrm{PRCC}_{\lambda}\approx0.857.
\end{equation}
In contrast, vaccination and treatment parameters exhibit negative associations with $\mathcal{R}_0$. These global sensitivity results are consistent with the analytical sensitivity indices.

\begin{figure}[htbp]
\centering
\begin{subfigure}[b]{0.4\textwidth}
\centering
\includegraphics[width=\textwidth]{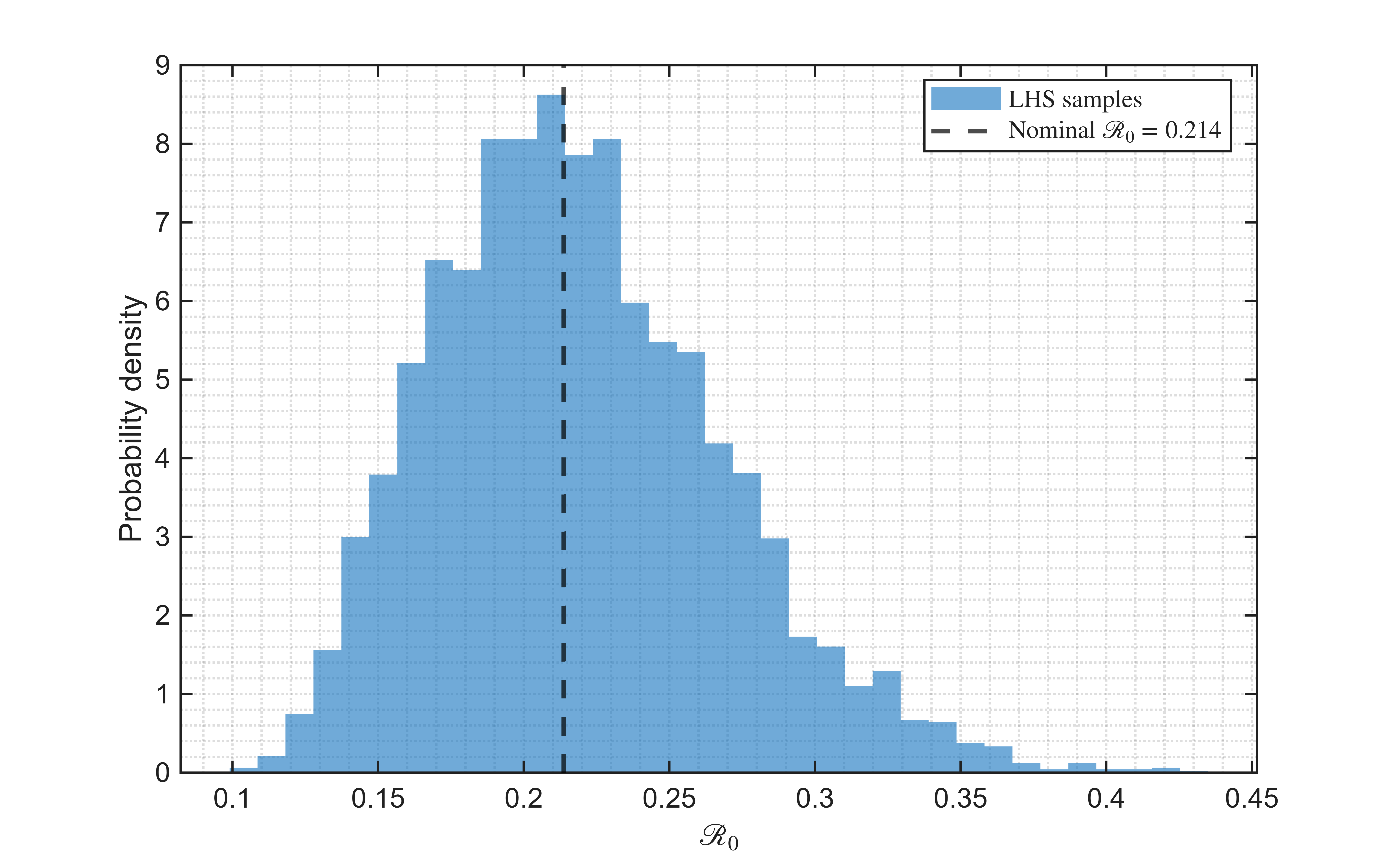}
\caption{LHS distribution of $\mathcal{R}_0$.}
\label{fig:LHS_R0}
\end{subfigure}
\hfill
\begin{subfigure}[b]{0.4\textwidth}
\centering
\includegraphics[width=\textwidth]{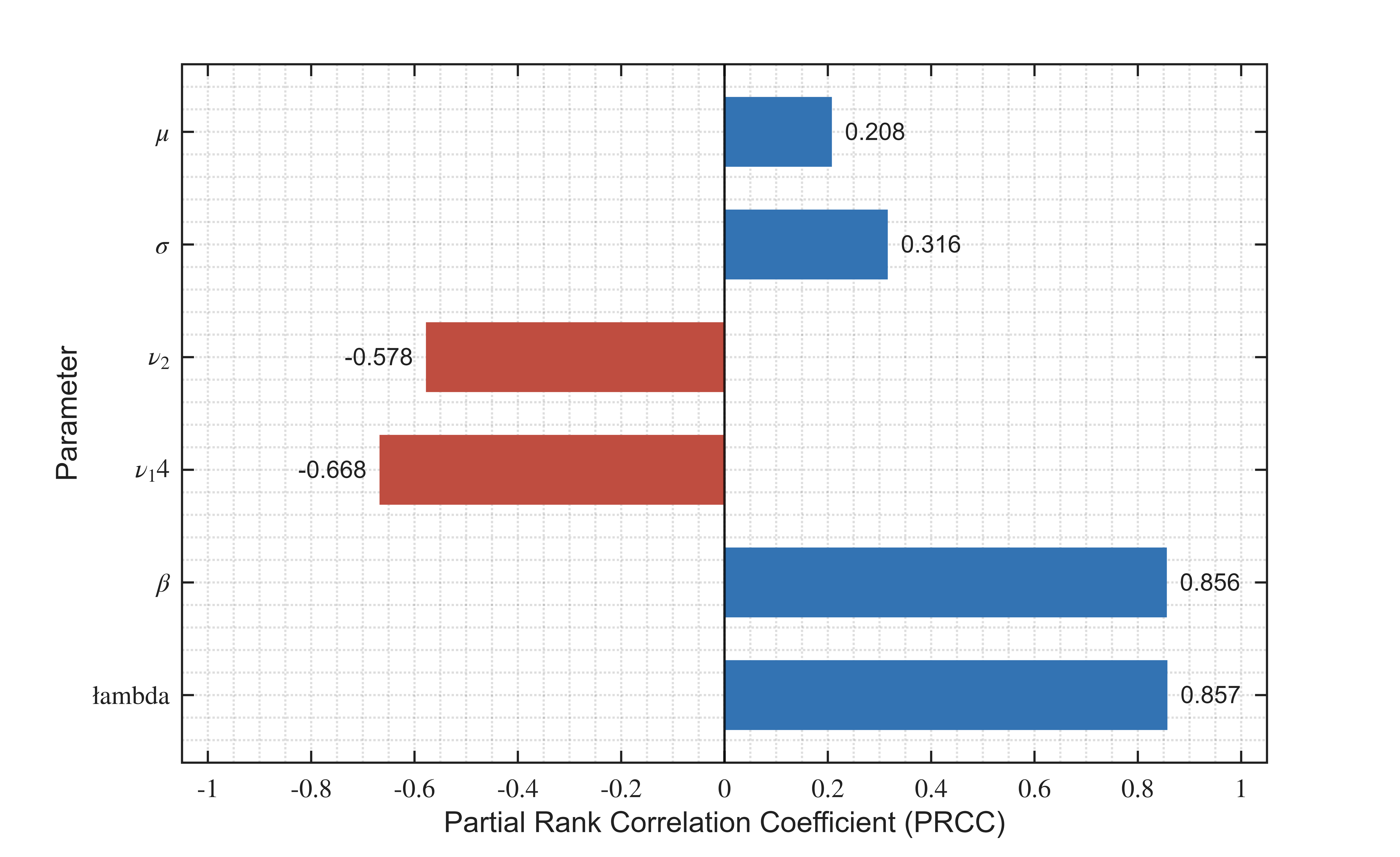}
\caption{PRCC analysis of $\mathcal{R}_0$.}
\label{fig:PRCC_R0_global}
\end{subfigure}
\caption{Global sensitivity analysis of the basic reproduction number $\mathcal{R}_0$. Panel (a) presents the distribution of $\mathcal{R}_0$ generated using Latin hypercube sampling, while panel (b) shows the corresponding partial rank correlation coefficients.}
\label{fig:global_sensitivity_R0}
\end{figure}

\section{Data-driven parameterization and validation}
\label{sec_data_parameterization}

To establish a quantitative connection between the proposed HBV transmission model and the epidemiological burden of hepatitis B in India, publicly available estimates from the Institute for Health Metrics and Evaluation (IHME) Global Burden of Disease (GBD) \cite{IHME2025HBVincidence} study and Our World in Data \cite{OrtizOspinaRoser2016HBV} are incorporated into the parameterization and validation procedure. The GBD estimates provide population-level information on hepatitis B incidence and mortality, while demographic information constrains the demographic component of the model \cite{World2024global,epoch2025major}. The Indian epidemiological quantities considered in the analysis are summarized in Table~\ref{tab:india_hbv_data}.

\begin{table}[htbp]
\centering
\begin{tabular}{lccc}
\hline
Indicator & Estimate & Year & Unit\\
\hline
Hepatitis B deaths & 80,910 & 2023 & deaths\\
Hepatitis B death rate & 5.61 & 2023 & per 100,000 population\\
Annual incident cases & 8,699,537 & 2023 & cases/year\\
Incidence rate & 602.72 & 2023 & per 100,000 population\\
Total population & 1,443,385,420 & 2023 & persons\\
Life expectancy at birth & 71.56 & 2023 & years\\
Chronic Hepatitis B and C mortality & 8.74 & 2022 & per 100,000 population\\
\hline
\end{tabular}
\caption{India-specific hepatitis B and demographic indicators used for data-driven parameterization.}
\label{tab:india_hbv_data}
\end{table}

The aggregate GBD estimates are not treated as direct observations of individual transition rates. Instead, the epidemiological quantities are mapped onto the corresponding model observables. The calibration therefore estimates a biologically admissible parameter set that reproduces the available population-level targets while respecting the structural assumptions of the model.

\subsection{Conversion of epidemiological data to model-scale quantities}
\label{subsec:data_conversion}

Since the model is formulated in normalized population variables, the reported epidemiological quantities are converted into per-capita rates. The annual HBV incidence rate of $602.72$ per $100,000$ population gives

\begin{equation}
\mathcal{I}_{\mathrm{GBD}}
=
\frac{602.72}{100000}
=
0.0060272\ {\rm year}^{-1}.
\label{eq:gbd_incidence_rate}
\end{equation}

Similarly, the reported hepatitis B mortality rate of $5.61$ per $100,000$ population gives

\begin{equation}
\mathcal{D}_{\mathrm{GBD}}
=
\frac{5.61}{100000}
=
5.61\times10^{-5}\ {\rm year}^{-1}.
\label{eq:gbd_death_rate}
\end{equation}

The reported life expectancy at birth, $71.56$ years, provides a demographic reference, giving

\begin{equation}
\mu_{\mathrm{dem}}
=
\frac{1}{71.56}
\approx
0.0139743\ {\rm year}^{-1}.
\label{eq:life_expectancy_rate}
\end{equation}

This value is used as the demographic estimate for the natural mortality parameter $\mu$. Thus,

\begin{equation}
\mathcal{I}_{\mathrm{GBD}}=0.0060272\ {\rm year}^{-1},
\qquad
\mathcal{D}_{\mathrm{GBD}}=5.61\times10^{-5}\ {\rm year}^{-1},
\end{equation}

\begin{equation}
\mu=0.0139743\ {\rm year}^{-1}.
\end{equation}

These conversions place the epidemiological data on the same scale as the normalized mathematical model.

\subsection{Model observables and epidemiological targets}
\label{subsec:data_mapping}

The two principal epidemiological observables used for calibration are the incidence of new HBV infections and disease-induced mortality. In the proposed model,

\begin{equation}
\mathcal{I}_{\mathrm{model}}(t)
=
\beta S(t)I(t),
\label{eq:model_incidence}
\end{equation}

where $\beta$ denotes the transmission parameter and $S(t)$ and $I(t)$ represent the susceptible and infected fractions, respectively. The disease-induced mortality flow is

\begin{equation}
\mathcal{D}_{\mathrm{model}}(t)
=
\sigma I(t),
\label{eq:model_death}
\end{equation}

where $\sigma$ is the disease-associated mortality parameter.

Accordingly, the calibration seeks parameter combinations for which

\begin{equation}
\mathcal{I}_{\mathrm{model}}
\approx
\mathcal{I}_{\mathrm{GBD}}
=
0.0060272\ {\rm year}^{-1},
\end{equation}

and

\begin{equation}
\mathcal{D}_{\mathrm{model}}
\approx
\mathcal{D}_{\mathrm{GBD}}
=
5.61\times10^{-5}\ {\rm year}^{-1}.
\end{equation}

In addition, HBsAg prevalence estimates for 2015 and 2020 are incorporated as historical prevalence constraints. The 2015 infected fraction is initialized consistently with the reported prevalence, while the model trajectory is required to reproduce the 2020 prevalence estimate.

The calibration therefore combines four epidemiological quantities:
\begin{enumerate}
    \item HBsAg prevalence in 2015,
    \item HBsAg prevalence in 2020,
    \item annual HBV incidence in 2023, and
    \item HBV mortality in 2023.
\end{enumerate}

\subsection{Calibration formulation}
\label{subsec:calibration_formulation}

Let the vector of estimated parameters be

\begin{equation}
\boldsymbol{\theta}
=
(\lambda,\beta,\nu_1,\nu_2,\sigma,R_{2015}),
\end{equation}

where $\lambda$ is the recruitment rate, $\beta$ is the transmission parameter, $\nu_1$ is the vaccination rate, $\nu_2$ is the treatment rate, $\sigma$ is the disease-induced mortality rate, and $R_{2015}$ denotes the recovered fraction at the beginning of the calibration period. The natural mortality rate $\mu$ is determined independently from the demographic reference.

The initial susceptible fraction is determined from the population normalization condition,

\begin{equation}
S_{2015}
=
1-I_{2015}-R_{2015},
\label{eq:S2015_constraint}
\end{equation}

where

\begin{equation}
I_{2015}=0.017.
\end{equation}

The parameter estimation problem is formulated as the constrained nonlinear optimization problem

\begin{equation}
\boldsymbol{\theta}^{*}
=
\underset{\boldsymbol{\theta}\in\Theta}{\operatorname{arg\,min}}
\;J(\boldsymbol{\theta}),
\label{eq:parameter_estimation}
\end{equation}

where $\Theta$ denotes the biologically admissible parameter domain. A normalized least-squares objective is used,

\begin{equation}
J(\boldsymbol{\theta})
=
\sum_{j=1}^{m}
w_j
\left[
\frac{Y_j^{\mathrm{model}}(\boldsymbol{\theta})
-
Y_j^{\mathrm{data}}}
{Y_j^{\mathrm{data}}}
\right]^2,
\label{eq:calibration_objective}
\end{equation}

where $Y_j$ denotes the selected epidemiological observable and $w_j$ is its corresponding weight. The constrained optimization is solved using the sequential quadratic programming implementation of \texttt{fmincon} in MATLAB, with parameter bounds imposed to prevent biologically inadmissible solutions.

\subsection{Calibrated parameter set}
\label{subsec:calibrated_parameters}

The best-fitting parameter set obtained from the constrained optimization is

\begin{equation}
\mu=0.0139742873,
\end{equation}

\begin{equation}
\lambda=0.0145414309,
\qquad
\beta=0.5098268496,
\end{equation}

\begin{equation}
\nu_1=0.0050586726,
\qquad
\nu_2=0.3145988771,
\end{equation}

\begin{equation}
\sigma=0.0030898517.
\end{equation}

The corresponding initial recovered and susceptible fractions are

\begin{equation}
R_{2015}=0.2980581902,
\end{equation}

and

\begin{equation}
S_{2015}
=
1-I_{2015}-R_{2015}
=
0.6849418098.
\end{equation}

The resulting basic reproduction number is

\begin{equation}
\mathcal{R}_0=1.1744281971.
\label{eq:calibrated_R0}
\end{equation}

The optimized objective function and corresponding RMSE are

\begin{equation}
J_{\min}
=
1.0295696888\times10^{-13},
\end{equation}

and

\begin{equation}
\mathrm{RMSE}
=
2.09\times10^{-9}.
\end{equation}

The calibrated parameter values are summarized in Table~\ref{tab:calibrated_parameters}.

\begin{table}[htbp]
\centering
\begin{tabular}{lcc}
\hline
Parameter & Calibrated value & Interpretation\\
\hline
$\mu$ & $0.0139742873$ & Natural mortality rate\\
$\lambda$ & $0.0145414309$ & Recruitment rate\\
$\beta$ & $0.5098268496$ & Transmission parameter\\
$\nu_1$ & $0.0050586726$ & Vaccination rate\\
$\nu_2$ & $0.3145988771$ & Treatment rate\\
$\sigma$ & $0.0030898517$ & Disease-induced mortality rate\\
$R_{2015}$ & $0.2980581902$ & Initial recovered fraction\\
\hline
\end{tabular}
\caption{Calibrated parameter values obtained from the data-driven optimization.}
\label{tab:calibrated_parameters}
\end{table}

Using the calibrated parameter set, the temporal evolution of the susceptible, infected, and recovered population fractions is shown in Fig.~\ref{HBV_calibrated}.

\begin{figure}[ht]
    \centering
    \includegraphics[width=0.8\linewidth]{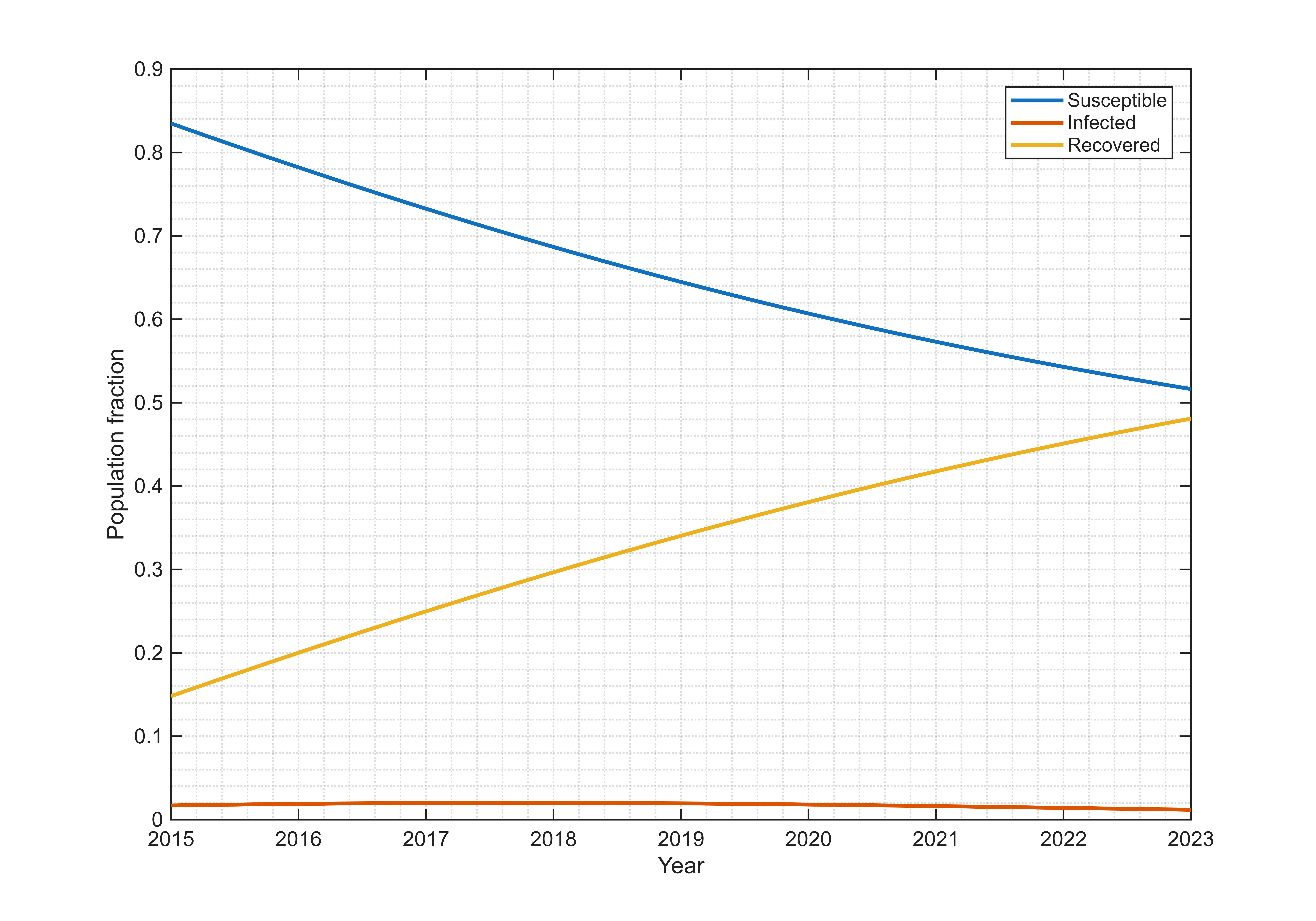}
    \caption{Temporal dynamics of the susceptible, infected, and recovered population fractions generated by the India-calibrated HBV model over the period 2015--2023.}
    \label{HBV_calibrated}
\end{figure}

\subsection{Model-data validation}
\label{subsec:model_data_validation}

The calibrated model reproduces the prescribed epidemiological targets with very small residual errors. The comparison between the epidemiological targets and calibrated model outputs is presented in Table~\ref{tab:model_data_comparison}.

\begin{table}[htbp]
\centering
\begin{tabular}{lccc}
\hline
Observable & Data & Model & Relative error (\%)\\
\hline
HBsAg prevalence, 2015
& $0.017000$
& $0.017000$
& $0.0000$\\

HBsAg prevalence, 2020
& $0.018000$
& $0.018000$
& $0.0000$\\

HBV incidence, 2023
& $0.00602720$
& $0.00602720$
& $0.0001$\\

HBV mortality, 2023
& $0.00005610$
& $0.00005610$
& $0.0000$\\
\hline
\end{tabular}
\caption{Comparison between the epidemiological targets and the corresponding calibrated model outputs.}
\label{tab:model_data_comparison}
\end{table}

As illustrated in Fig.~\ref{HBV_prevalence}, the calibrated trajectory closely matches the observed HBsAg prevalence values for 2015 and 2020.

\begin{figure}[htbp]
    \centering
    \includegraphics[width=0.6\linewidth]{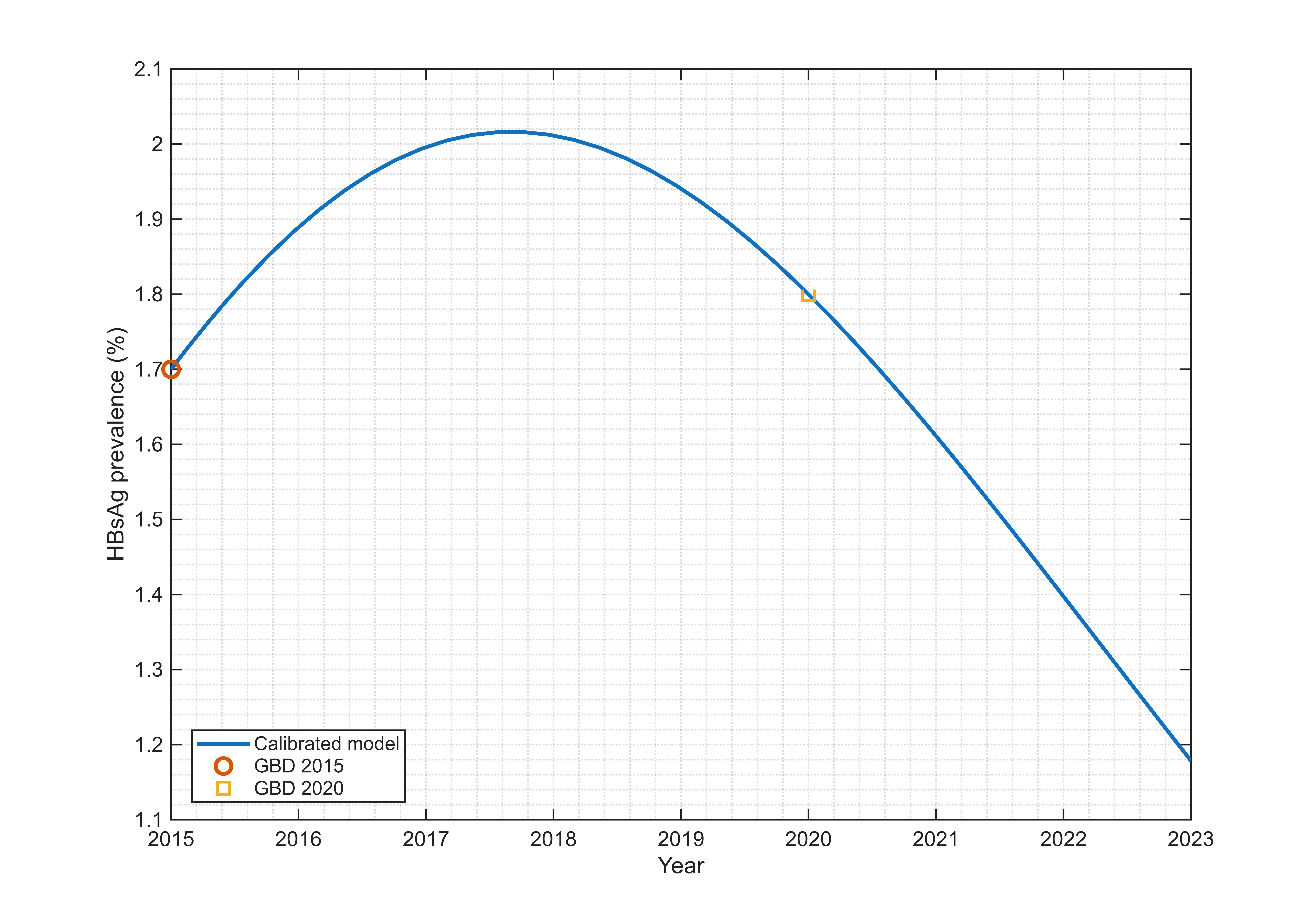}
    \caption{Comparison of the calibrated model-derived HBsAg prevalence with the corresponding GBD estimates for India in 2015 and 2020.}
    \label{HBV_prevalence}
\end{figure}

The calibrated parameter set simultaneously reproduces the available prevalence, incidence, and mortality targets, with the small value of the calibration objective indicating close agreement under the adopted objective function.

The calibrated model was further evaluated against the GBD 2023 estimates of HBV incidence and mortality. As shown in Fig.~\ref{Incidence_Mortality_comparison}, the calibrated model outputs agree with the corresponding GBD 2023 estimates for India.

\begin{figure}[htbp]
    \centering

    \begin{minipage}[b]{0.48\textwidth}
        \centering
        \includegraphics[width=\textwidth]{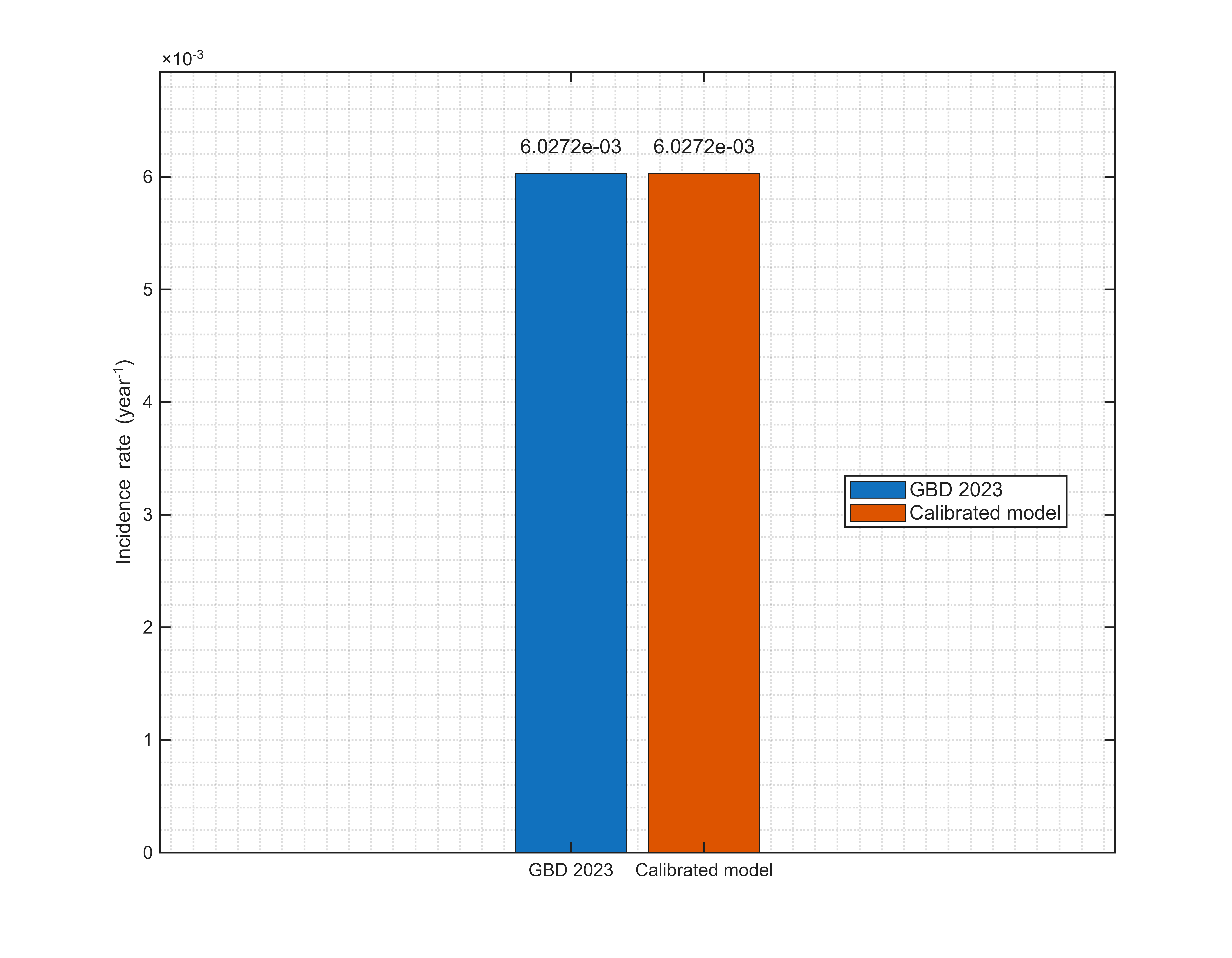}
        \caption*{(a) HBV incidence rate}
    \end{minipage}
    \hfill
    \begin{minipage}[b]{0.48\textwidth}
        \centering
        \includegraphics[width=\textwidth]{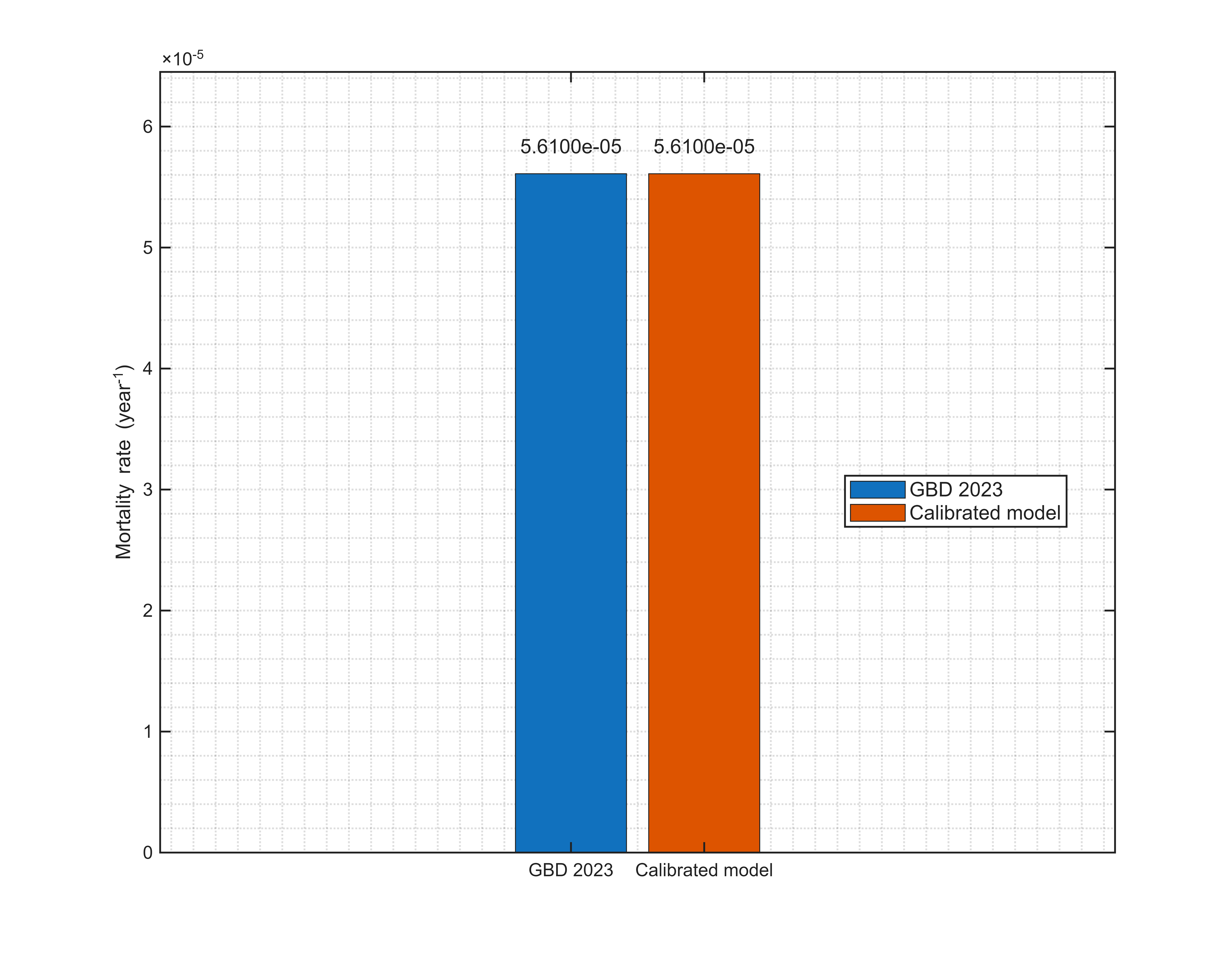}
        \caption*{(b) HBV mortality rate}
    \end{minipage}

    \caption{Comparison of the calibrated HBV model with GBD 2023 estimates for India: (a) HBV incidence rate and (b) HBV mortality rate.}
    \label{Incidence_Mortality_comparison}
\end{figure}

Nevertheless, a small optimization residual does not by itself imply that every individual parameter is uniquely identifiable. With aggregate epidemiological observations, different combinations of transmission, vaccination, treatment, and mortality parameters may produce similar model outputs. Therefore, practical identifiability is assessed using profile-based analysis, together with the robustness of $\mathcal{R}_0$ across the resulting parameter ranges.

\subsection{Multi-start calibration and solution robustness}
\label{subsec:multistart}

To assess the dependence of the calibration results on the initial parameter guess, a multi-start calibration procedure was performed. Six distinct admissible initial parameter sets were selected, and the constrained nonlinear optimization problem was solved independently from each starting point.

The resulting objective function values were

\begin{equation}
\begin{aligned}
J_1 &= 8.481253\times10^{-11},\\
J_2 &= 3.826401\times10^{-12},\\
J_3 &= 1.029570\times10^{-13},\\
J_4 &= 2.207182\times10^{-11},\\
J_5 &= 9.671759\times10^{-11},\\
J_6 &= 2.305351\times10^{-11}.
\end{aligned}
\end{equation}

All six runs converged to parameter combinations producing very small objective function values, although the resulting parameter vectors were not identical. In particular, the estimated transmission parameter $\beta$ varied from approximately $0.356$ to $1.451$, while the treatment rate $\nu_2$ varied from approximately $0.307$ to $0.795$.

These results indicate that multiple parameter combinations can reproduce the available epidemiological observations with similar accuracy, consistent with parameter compensation. This motivates the profile-based practical identifiability analysis in the following subsection.

\subsection{Practical identifiability analysis}
\label{subsec:practical_identifiability}

The multi-start calibration demonstrates that highly accurate fits can be obtained from different initial parameter guesses but does not quantify the extent to which the available observations constrain individual model parameters. Therefore, a profile-based practical identifiability analysis was performed \cite{dankwa2022structural,eisenberg2013identifiability,kao2018practical,tuncer2018structural}.

For each parameter of interest, the parameter was fixed sequentially over its prescribed admissible range, while all remaining calibrated parameters were re-optimized. This generates the profile objective function $J_{\mathrm{prof}}(\theta)$. A parameter value was considered practically compatible with the available observations when

\begin{equation}
J_{\mathrm{threshold}}=10J_{\min},
\end{equation}

where

\begin{equation}
J_{\min}
=
4.024540117879\times10^{-8}.
\end{equation}

Consequently,

\begin{equation}
J_{\mathrm{threshold}}
=
4.024540117879\times10^{-7}.
\end{equation}

The recruitment parameter $\lambda$ was not included in the profile analysis because it was treated as a demographically constrained parameter. The analysis was therefore restricted to $\beta$, $\nu_1$, $\nu_2$, $\sigma$, and $R_{2015}$.

The profile-based results are illustrated in Fig.~\ref{profile_identifiability}. Each panel shows the minimized objective function as the corresponding parameter is varied, with the remaining parameters re-optimized. The horizontal dashed line represents $J_{\mathrm{threshold}}$, while the vertical dotted line indicates the best-fitting value.

\begin{figure}[ht]
    \centering
    \includegraphics[width=0.6\linewidth]{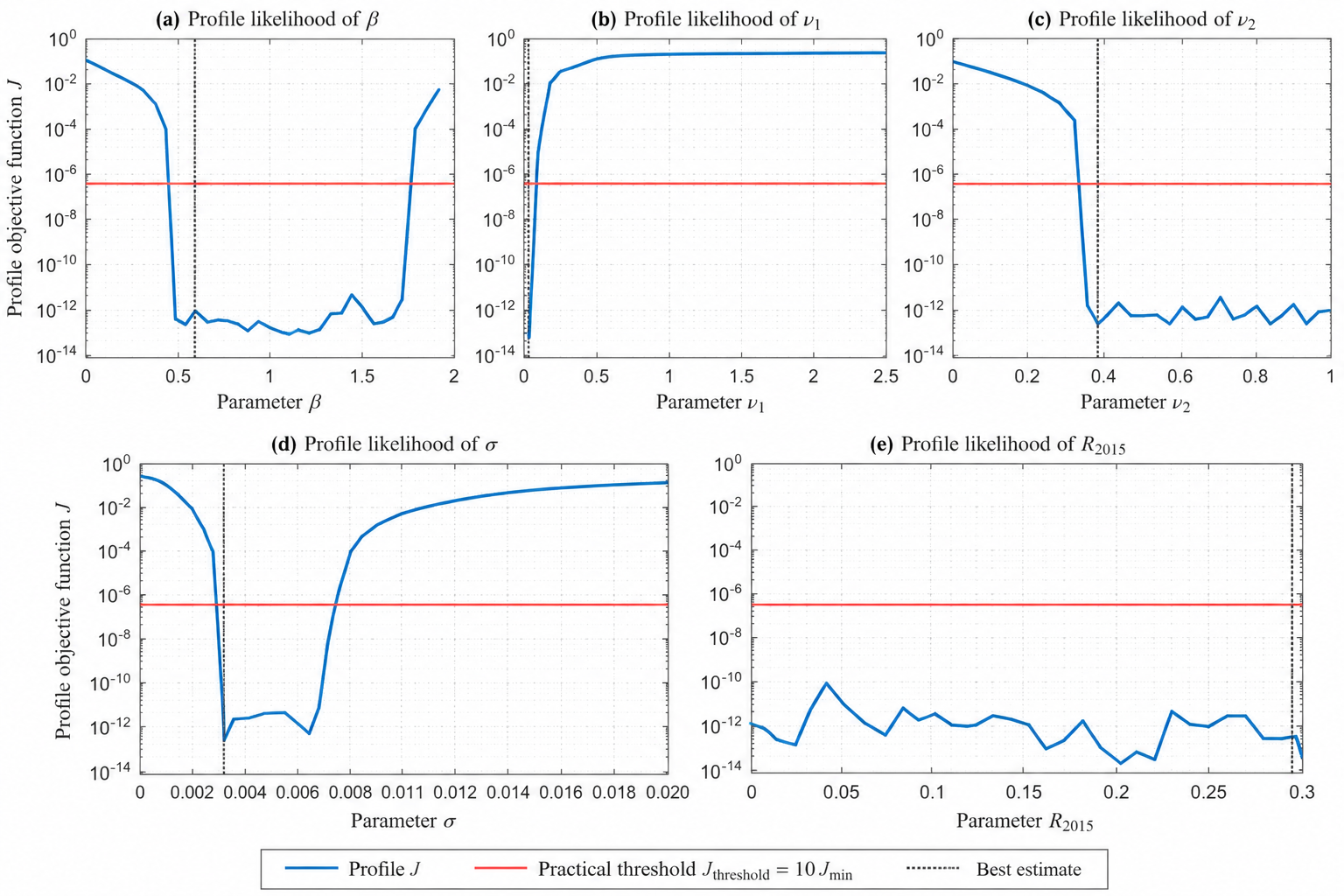}
    \caption{Profile-based practical identifiability analysis of the calibrated HBV model parameters. The minimized objective function $J$ is shown for each profiled parameter, with the remaining parameters re-optimized. The horizontal dashed line denotes the practical threshold $J_{\mathrm{threshold}}=10J_{\min}$, while the vertical dotted line indicates the best-fitting parameter estimate.}
    \label{profile_identifiability}
\end{figure}

The resulting practically compatible parameter ranges are summarized in Table~\ref{tab:identifiability_ranges}.

\begin{table}[ht]
\centering
\begin{tabular}{lcc}
\hline
Parameter & Lower bound & Upper bound\\
\hline
$\beta$ & $0.40080000$ & $1.80010000$\\
$\nu_1$ & $0.00100000$ & $0.00505867$\\
$\nu_2$ & $0.30070000$ & $1.00000000$\\
$\sigma$ & $0.00308985$ & $0.00600070$\\
$R_{2015}$ & $0.00000000$ & $0.30000000$\\
\hline
\end{tabular}
\caption{Practically compatible parameter ranges obtained from the profile-based identifiability analysis.}
\label{tab:identifiability_ranges}
\end{table}

The profiles reveal different levels of practical parameter information. The transmission parameter $\beta$ exhibits the broad profile-compatible interval

\[
0.4008\leq\beta\leq1.8001,
\]

while the treatment rate satisfies

\[
0.3007\leq\nu_2\leq1.0000.
\]

The vaccination rate has the comparatively narrower compatible interval

\[
0.0010\leq\nu_1\leq0.00505867,
\]

whereas the initial infected fraction $R_{2015}$ remains compatible throughout

\[
0\leq R_{2015}\leq0.30.
\]

For the disease-induced mortality parameter,

\[
0.00308985\leq\sigma\leq0.00600070
\]

remains practically compatible with the observations.

Overall, the profile analysis demonstrates that the epidemiological observations strongly constrain the model outputs while providing weaker constraints on several individual mechanistic parameters. The broad profile-compatible regions indicate parameter compensation, whereby changes in one parameter can be offset by re-optimization of the remaining parameters. Consequently, the calibrated parameter vector should be interpreted as a representative epidemiologically consistent parameterization rather than as a uniquely identifiable estimate of every underlying biological transition rate.

\subsection{Robustness of the basic reproduction number}
\label{subsec_R0_robustness}

Since the basic reproduction number $\mathcal{R}_0$ is a central threshold
quantity in the subsequent dynamical analysis, its robustness with respect
to the practically compatible parameter ranges was examined. For each
parameter profile considered in the practical identifiability analysis,
the corresponding value of $\mathcal{R}_0$ was evaluated at each practically
compatible parameter value, using the remaining parameters at their
profile-optimized values. The resulting ranges are summarized in
Table~\ref{tab:R0_robustness}.

\begin{table}[ht]
\centering
\begin{tabular}{lcc}
\hline
Profiled parameter & Minimum $\mathcal{R}_0$ & Maximum $\mathcal{R}_0$\\
\hline
$\beta$ & $0.30572006$ & $1.18328873$\\
$\nu_1$ & $1.17516722$ & $1.47418921$\\
$\nu_2$ & $0.23144022$ & $1.41587808$\\
$\sigma$ & $0.27525802$ & $1.17533417$\\
$R_{2015}$ & $0.25833621$ & $1.39411413$\\
\hline
\end{tabular}
\caption{Ranges of the basic reproduction number $\mathcal{R}_0$ obtained over the practically compatible parameter profiles.}
\label{tab:R0_robustness}
\end{table}

The robustness results demonstrate that the inferred value of
$\mathcal{R}_0$ varies across the practically compatible parameter space.
The profiles associated with $\beta$, $\nu_2$, $\sigma$, and $R_{2015}$
permit values of $\mathcal{R}_0$ both below and above the epidemic
threshold $\mathcal{R}_0=1$. Thus, these four profiles correspond to both
sub-threshold and super-threshold transmission regimes within the
practically compatible parameter space.

In contrast, the practically compatible $\nu_1$ profile gives

\begin{equation}
1.17516722
\leq
\mathcal{R}_0
\leq
1.47418921,
\end{equation}

so that the entire corresponding range remains above unity.

For the best-fitting calibrated parameter vector, the model gives

\begin{equation}
\mathcal{R}_0^{*}
=
1.174428194141
>
1.
\end{equation}

Therefore, the reference parameterization corresponds to a super-threshold
transmission regime. However, the profile analysis demonstrates that this
conclusion is not uniformly preserved throughout the practically
compatible parameter space.

The robustness analysis is illustrated in
Fig.~\ref{R0_robustness_profiles}, where the individual $\mathcal{R}_0$
profiles and the corresponding minimum--maximum ranges are presented. The
horizontal dashed line denotes the epidemic threshold
$\mathcal{R}_0=1$, while the vertical reference line indicates the
best-fitting value of the profiled parameter.

\begin{figure}[ht]
    \centering
    \includegraphics[width=0.85\linewidth]{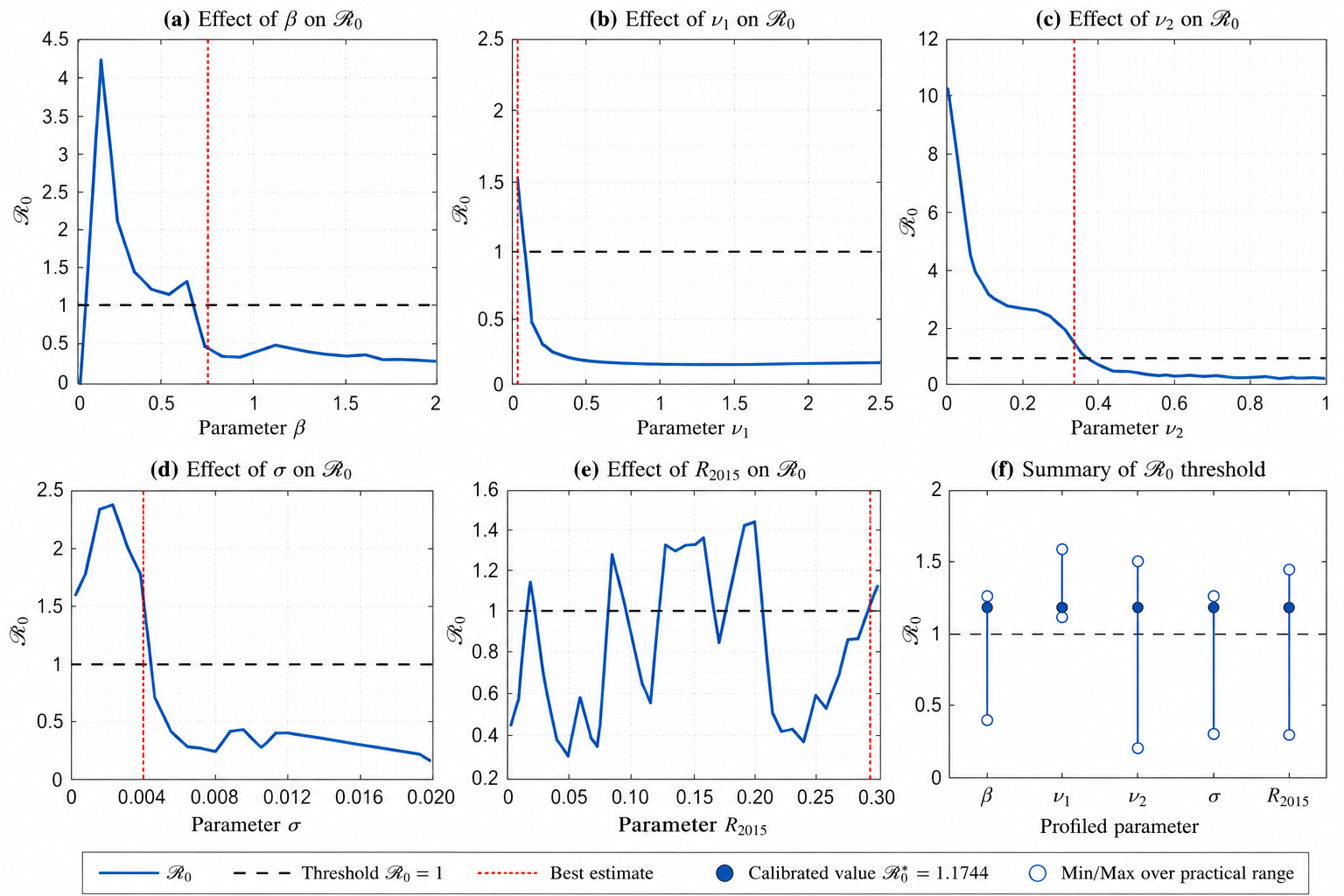}
    \caption{Robustness of the basic reproduction number $\mathcal{R}_0$ over the practically compatible parameter profiles.}
    \label{R0_robustness_profiles}
\end{figure}

The results indicate that the numerical value
$\mathcal{R}_0^{*}=1.1744$ should be interpreted as the threshold quantity
associated with the selected reference calibration rather than as a
uniquely determined epidemiological quantity. The $\nu_1$ profile provides
a comparatively robust qualitative result because its entire practically
compatible range remains above the epidemic threshold.

Overall, uncertainty in the calibrated parameter vector can propagate into
uncertainty in the inferred transmission regime. Accordingly, subsequent
threshold-dependent conclusions are interpreted with respect to the
reference calibrated parameterization, while the profile-based
$\mathcal{R}_0$ ranges provide an uncertainty framework for the resulting
dynamical conclusions.

\subsection{Summary of the data-driven parameterization}
\label{subsec:data_parameterization_summary}

The data-driven analysis establishes a quantitative parameterization of the
proposed HBV model using India-specific epidemiological and demographic
information. The available incidence, mortality, and HBsAg prevalence
estimates were incorporated as epidemiological targets in a constrained
nonlinear calibration procedure. Multi-start calibration and profile-based
practical identifiability analysis were then used to assess the extent to
which the available aggregate data constrain the model parameters.

The reference calibration yields the minimum objective function

\begin{equation}
J_{\min}
=
4.024540117879\times10^{-8},
\end{equation}

with the calibrated parameter values

\begin{equation}
\begin{aligned}
\beta^* &= 0.5098268496,\\
\nu_1^* &= 0.0050586726,\\
\nu_2^* &= 0.3145988771,\\
\sigma^* &= 0.0030898517,\\
R_{2015}^* &= 0.2980581902.
\end{aligned}
\end{equation}

The calibration provides close simultaneous agreement with the
epidemiological targets. However, the multi-start and profile-based
analyses demonstrate that similarly accurate fits can be obtained from
different parameter combinations and that the practically compatible
parameter ranges differ substantially in width. Thus, the available
aggregate epidemiological data strongly constrain the model outputs but do
not uniquely determine all underlying mechanistic parameters.

For the reference calibrated parameter vector,

\begin{equation}
\mathcal{R}_0^*
=
1.174428194141
>
1.
\end{equation}

Hence, the reference parameterization corresponds to a super-threshold
transmission regime. Nevertheless, the $\mathcal{R}_0$ robustness analysis
shows that practically compatible parameter combinations can generate both
sub-threshold and super-threshold values of $\mathcal{R}_0$, whereas the
entire practically compatible $\nu_1$ profile remains above unity.

Accordingly, the calibrated parameter set is used as the reference
parameterization for the subsequent dynamical analysis, while the
practical parameter ranges and corresponding robustness of $\mathcal{R}_0$
are retained when interpreting threshold-dependent results.

\section{Discussion}
\label{sec_discussion}

The present study developed and analyzed a vaccination-treatment model for
HBV transmission to investigate how transmission and intervention
mechanisms jointly determine disease persistence. The analytical and
numerical results provide a consistent description of the epidemic
threshold, equilibrium structure, bifurcation behavior, and intervention
effects. In particular, the basic reproduction number $\mathcal{R}_0$
provides the fundamental threshold criterion, while the equilibrium and
bifurcation analyses describe the system behavior as this threshold is
crossed.

The equilibrium and bifurcation analyses demonstrate that the disease-free
and endemic states are governed by the threshold condition
$\mathcal{R}_0=1$. The transition through this threshold occurs through a
forward transcritical bifurcation with respect to the transmission
parameter $\beta$. This result is independently supported by the numerical
equilibrium branches and Jacobian eigenvalue analysis, which demonstrate
the stability exchange between the disease-free and endemic equilibria at
the critical transmission level.

The two-parameter analysis further characterizes the threshold structure by
considering transmission and vaccination simultaneously. The curve
$\mathcal{R}_0(\beta,\nu_1)=1$ identifies combinations of transmission
intensity and vaccination separating sub-threshold and super-threshold
regimes. The corresponding endemic-equilibrium analysis shows that
vaccination can reduce the magnitude of persistent infection even when
transmission remains sufficiently high for an endemic state to exist.
Thus, vaccination influences both the epidemic threshold and the endemic
infection burden.

The sensitivity analyses provide further insight into the mechanisms
underlying these threshold changes. Transmission-related parameters
contribute positively to epidemic persistence, whereas vaccination and
treatment act in the direction of disease suppression. The agreement
between the local normalized sensitivity analysis and the global
LHS-PRCC analysis indicates that these qualitative effects are maintained
beyond the single baseline parameterization.

An important aspect of the study is the integration of the analytical model
with India-specific epidemiological information. The data-driven
calibration produced close correspondence between the model outputs and
the prescribed HBsAg prevalence, incidence, and mortality targets. The
multi-start analysis showed that highly accurate fits could be obtained
from different initial parameter values, although the resulting parameter
combinations were not identical. Thus, a close fit to aggregate
epidemiological observations does not necessarily imply unique
determination of the underlying mechanistic parameters.

The profile-based practical identifiability analysis further demonstrated
substantial parameter compensation, whereby changes in one parameter can
be offset by re-optimization of other parameters while maintaining a
comparably good fit to the available observations. The relatively
localized profile for $\sigma$ indicates that this parameter is more
strongly constrained by the available targets. Overall, the calibrated
model is well supported at the level of epidemiological observables,
whereas greater caution is required when assigning biological
interpretations to individual parameter estimates.

The practical identifiability results also affect the interpretation of
the reproduction number. The reference calibrated parameterization gives

\begin{equation}
\mathcal{R}_0^*=1.1744>1,
\end{equation}

corresponding to a super-threshold regime. However, practically compatible
parameter combinations can generate values of $\mathcal{R}_0$ both below
and above unity. Therefore, $\mathcal{R}_0^*$ should be interpreted as the
reproduction number associated with the selected reference
parameterization rather than as a uniquely determined population-level
constant. The profile analysis instead quantifies the range of threshold
behavior compatible with the available epidemiological information.

From an intervention perspective, the combined analytical and data-driven
results indicate that vaccination and treatment should be considered in
relation to the prevailing transmission environment. The two-parameter
threshold curve demonstrates the relationship between transmission and
vaccination required to maintain a sub-threshold regime, while the
endemic-equilibrium and sensitivity analyses show that intervention
intensity also affects the magnitude of persistent infection. These
results provide a mathematical basis for examining combinations of
prevention and treatment strategies rather than focusing exclusively on a
single intervention.

Several limitations should nevertheless be acknowledged. The model is
deterministic and represents population-level average dynamics and
therefore does not explicitly account for stochastic demographic effects,
age-dependent transmission, spatial heterogeneity, behavioral variation,
or more detailed stages of HBV infection. In addition, the calibration
relies on a limited set of aggregate epidemiological targets, contributing
to the practical uncertainty observed in several mechanistic parameters.
The resulting parameter estimates and threshold values should therefore be
interpreted as model-specific quantities associated with the selected
epidemiological setting rather than as universal biological constants.

These limitations provide several directions for future investigation.
Incorporating age structure, time-dependent vaccination and treatment,
spatial heterogeneity, stochastic effects, and more extensive longitudinal
HBV surveillance data could improve parameter determination and allow more
detailed population-specific predictions. Additional prevalence,
incidence, treatment, vaccination, and mortality observations would also
be valuable for reducing parameter compensation and strengthening the
empirical basis of the threshold analysis.

Overall, the study combines mathematical analysis with data-driven
parameterization to describe HBV persistence through epidemic thresholds,
equilibrium stability, bifurcation, intervention geometry, sensitivity,
calibration, and practical identifiability, while explicitly recognizing
the uncertainty associated with limited epidemiological information.

\section{Conclusions}
\label{sec_conclusion}

In this study, we developed and rigorously analyzed a vaccination-treatment
model for HBV transmission, with emphasis on the threshold, equilibrium,
and bifurcation structure governing disease persistence. The disease-free
and endemic equilibria were characterized, and the basic reproduction
number $\mathcal{R}_0$ was derived as the principal threshold quantity.
The analytical results established a forward transcritical bifurcation at
$\mathcal{R}_0=1$ with respect to the transmission parameter $\beta$,
while numerical equilibrium continuation and eigenvalue calculations
confirmed the predicted stability transition.

A central contribution is the extension of the threshold analysis to a
two-parameter geometry. By simultaneously varying transmission and
vaccination, the threshold curve
$\mathcal{R}_0(\beta,\nu_1)=1$ divides the parameter plane into
sub-threshold and super-threshold regions. The corresponding
endemic-equilibrium analysis and sensitivity results further demonstrate
the influence of vaccination, treatment, and transmission on persistent
infection.

The model was further parameterized using India-specific epidemiological
information, including HBsAg prevalence, HBV incidence, and mortality
estimates. The reference calibration achieved

\begin{equation}
J_{\min}=4.024540117879\times10^{-8},
\end{equation}

indicating close agreement between the calibrated model and the prescribed
epidemiological targets. Multi-start calibration and profile-based
practical identifiability analysis showed that similarly accurate fits can
be obtained from different parameter combinations, with substantial
uncertainty remaining in several individual mechanistic parameters.

For the reference calibrated parameterization,

\begin{equation}
\mathcal{R}_0^*=1.174428194141>1,
\end{equation}

corresponding to a super-threshold transmission regime. However, the
robustness analysis demonstrated that practically compatible parameter
combinations can produce both sub-threshold and super-threshold values of
$\mathcal{R}_0$. Thus, the reference value
$\mathcal{R}_0^*=1.1744$ should be interpreted as the reproduction number
associated with the selected reference parameterization rather than as a
uniquely determined epidemiological quantity.

Overall, the study provides an integrated mathematical and data-driven
framework connecting epidemic thresholds, equilibrium stability, forward
transcritical bifurcation, two-parameter intervention geometry,
sensitivity, calibration, and practical identifiability. The results
demonstrate the influence of vaccination and treatment on HBV transmission
and persistence while highlighting the need for additional epidemiological
information to more precisely constrain several underlying mechanistic
parameters. The framework provides a basis for further investigation of
population-specific HBV control strategies under epidemiological and
parameter uncertainty.

\bibliographystyle{ieeetr}
\bibliography{biblog}
\end{document}